\documentclass{amsart}
\usepackage{amscd}
\usepackage{amsmath,empheq}
\usepackage{amsfonts}
\usepackage{amssymb}
\usepackage{mathrsfs}
\usepackage[all]{xy}
\newtheorem{theorem}{Theorem}

\newtheorem{lemma}[theorem]{Lemma}
\newtheorem{prop}[theorem]{Proposition}
\newtheorem{remark}{Remark}
\newtheorem{corollary}[theorem]{Corollary}
\newtheorem{claim}{Claim}

\newenvironment{proof-sketch}{\noindent{\bf Sketch of Proof}\hspace*{1em}}{\qed\bigskip}

\everymath{\displaystyle}
\newcommand{\RR}{\mathbb R}
\newcommand{\NN}{\mathbb N}

\newcommand{\ZZ}{\mathbb Z}

\renewcommand{\leq}{\leqslant}

\renewcommand{\geq}{\geqslant}
\begin{document}

\title[$(p,2)$--equations asymmetric at both zero and infinity]{$(p,2)$--equations asymmetric at both zero and infinity}
\author[N.S. Papageorgiou]{Nikolaos S. Papageorgiou}
\address[N.S. Papageorgiou]{National Technical University, Department of Mathematics,
				Zografou Campus, Athens 15780, Greece \&  Institute of Mathematics, Physics and Mechanics,  1000 Ljubljana, Slovenia}
\email{\tt npapg@math.ntua.gr}
\author[V.D. R\u{a}dulescu]{Vicen\c{t}iu D. R\u{a}dulescu}
\address[V.D. R\u{a}dulescu]{Institute of Mathematics ``Simion Stoilow" of the Romanian Academy, P.O. Box 1-764, 014700 Bucharest, Romania \& Department of Mathematics, University of Craiova,  200585 Craiova, Romania \&  Institute of Mathematics, Physics and Mechanics,  1000 Ljubljana, Slovenia}
\email{\tt vicentiu.radulescu@imar.ro}
\author[D.D. Repov\v{s}]{Du\v{s}an D. Repov\v{s}}
\address[D.D. Repov\v{s}]{Faculty of Education and Faculty of Mathematics and Physics, University of Ljubljana  \&  Institute of Mathematics, Physics and Mechanics, 1000 Ljubljana, Slovenia}
\email{\tt dusan.repovs@guest.arnes.si}
\keywords{$p$-Laplacian, asymmetric reaction, resonance, Fu\v{c}ik spectrum, constant sign solutions, nodal solution, critical groups, Morse relation.\\
\phantom{aa} 2010 AMS Subject Classification: Primary: 35J20. Secondary: 35J60, 58E05}
\begin{abstract}
We consider a $(p,2)$--equation, that is, a nonlinear nonhomogeneous elliptic equation driven by the sum of a $p$-Laplacian and a Laplacian with $p>2$. The reaction term is $(p-1)$--linear but exhibits asymmetric behaviour at $\pm\infty$ and at $0^{\pm}$. Using variational tools, together with truncation and comparison techniques and Morse theory, we prove two multiplicity theorems, one of them providing sign information for all the solutions (positive, negative, nodal).
\end{abstract}
\maketitle

\section{Introduction}
Let $\Omega\subseteq\RR^\NN$ be a bounded domain with a $C^2$-boundary $\partial\Omega$. In this paper we study the following nonlinear nonhomogeneous Dirichlet problem
\begin{equation}\label{eq1}
	-\Delta_p u(z)-\Delta u(z) = f(z,u(z))\ \mbox{in}\ \Omega\ ,\ u|_{\partial\Omega}=0 ,\ 2<p.
\end{equation}

Here, $\Delta_p$ denotes the $p$-Laplace differential operator defined by
$$\Delta_p u(z) = {\rm div}\,(|Du|^{p-2}Du)\ \mbox{for all}\ u\in W^{1,p}_0(\Omega).$$

If $p=2$, then $\Delta_2=\Delta$ the Laplacian.

In problem (\ref{eq1}) the reaction term $f(z,x)$ is a Carath\'eodory function such that $f(z,0)=0$. We assume that $f(z,\cdot )$ exhibits $(p-1)$-linear growth near $\pm\infty$. However, the growth of $f(z,\cdot)$ is asymmetric near $\pm\infty$. More precisely, the quotient $\frac{f(z,x)}{|x|^{p-2}x}$ crosses at least the principal eigenvalue $\hat{\lambda}_1(p)>0$ of $(-\Delta_p,W^{1,p}_0(\Omega))$ as we move from $-\infty$ to $+\infty$ (crossing or jumping nonlinearity). In the negative direction we allow resonance with respect to $\hat{\lambda}_1(p)>0$, while in the positive direction resonance can occur with respect of any nonprincipal eigenvalue of $(-\Delta_p,W^{1,p}_0(\Omega))$. A similar asymmetric behaviour we have as $x\rightarrow 0^{\pm}$. This time the quotient $\frac{f(z,x)}{x}$ crosses $\hat{\lambda}_1(2)>0$. Under this double asymmetric setting, we prove a multiplicity theorem producing three nontrivial smooth solutions and provide sign information for all of them. A second multiplicity theorem is also proved without sign information for the third solution.

Equations involving the sum of a Laplacian and a $p$-Laplacian, arise in problems of mathematical physics, see Cherfils \& Ilyasov \cite{9} (plasma physics) and Benci, D'Avenia, Fortunato \& Pisani \cite{6} (quantum physics). Recently, there have been existence and multiplicity results for different classes of such equations. We mention the works of Aizicovici, Papageorgiou \& Staicu \cite{3}, Cingolani \& Degiovanni \cite{10}, Gasinski \& Papageorgiou \cite{13,15}, Papageorgiou \& R\u{a}dulescu \cite{21,22}, Papageorgiou, R\u{a}dulescu \& Repov\v{s} \cite{25}, Sun \cite{30}, Sun, Zhang \& Su \cite{31}, Yang \& Bai \cite{32}. In the aforementioned works, only Papageorgiou \& R\u{a}dulescu \cite{22} dealt with an asymmetric $p$-sublinear reaction term. The authors in \cite{22} considered
 a reaction term $f(z,x)$ such that the quotient $\frac{f(z,x)}{|x|^{p-2}x}$ crosses only the first eigenvalue $\hat{\lambda}_1(p)$ as we move from $-\infty$ to $+\infty$ and resonance is allowed at $-\infty$. At zero the behaviour of the quotient $\frac{f(z,x)}{x}$ is symmetric. Finally, in \cite{22} the multiplicity result does not produce nodal solutions. Concerning asymmetric sublinear problems, we also mention the semilinear works of D'Agui, Marano \& Papageorgiou \cite{11} (Robin problems with an indefinite and unbounded potential), and Recova \& Rumbos \cite{28} (Dirichlet problems with zero potential).

Our approach is variational, based on the critical point theory, combined with a
suitable truncation and comparison techniques and Morse theory (critical groups).

\section{Mathematical Background}

Let $X$ be a Banach space and $X^{*}$ its topological dual. By $\left\langle \cdot,\cdot\right\rangle$ we denote the duality brackets for the pair $(X^{*}, X)$. Given $\varphi\in C^{1}(X,\RR)$, we say that $\varphi$ satisfies the ``Cerami condition" (the ``C-condition" for short), if the following holds:
\begin{center}
``Every sequence $\{u_n\}_{n\geq1}\subseteq X$ such that $\{\varphi(u_n)\}_{n\geq1}\subseteq\RR$ is bounded and
$$(1+||u_n||)\varphi^{'}(u_n)\rightarrow0\ \mbox{in}\ X^{*}\ \mbox{as}\ n\rightarrow\infty,$$
admits a strongly convergent subsequence".
\end{center}

This is a compactness-type condition on the functional. It leads to a deformation theorem from which one can derive the minimax theory of the critical values of $\varphi$. One of the main results in this theory is the so-called ``mountain pass theorem" of Ambrosetti \& Rabinowitz \cite{5}, stated here in a slightly more general form (see Gasinski \& Papageorgiou \cite{12}).

\begin{theorem}\label{th1}
	Assume that $\varphi\in C^{1}(X,\RR)$ satisfies the C-condition, $u_0,u_1\in X, ||u_1-u_0||>r>0$,
	$$\max\{\varphi(u_o), \varphi(u_1)\} < \inf{[\varphi(u):||u-u_0||=r]=m_r}$$
	and $c=\inf_{\gamma\in\Gamma}{\max_{0\leq t\leq1}{\varphi(\gamma(t))}}$ with $\Gamma=\{\gamma\in C([0,1],X): \gamma(0)=u_0, \gamma(1)=u_1\}$. Then $c\geq m_r$ and $c$ is a critical value of $\varphi$ (that is, there exists $u\in X$ such that $\varphi(u)=c$, $\varphi'(u)=0$).
\end{theorem}

In the study of (\ref{eq1}) we will use the Sobolev spaces $W^{1,p}_0(\Omega)$, $H^{1}_0(\Omega)$ and the Banach space $C^{1}_0(\overline{\Omega}) = \{u\in C^{1}(\overline{\Omega}): u|_{\partial\Omega}=0\}$.

By $||\cdot||$ we denote the norm of $W^{1,p}_0(\Omega)$. By Poincar\'e's inequality, the norm of $W^{1,p}_0(\Omega)$ can be defined by
$$||u||=||Du||_p\ \mbox{for all}\ u\in W^{1,p}_0(\Omega).$$

The Sobolev space $H^{1}_0(\Omega)$ is a Hilbert space and as above, the Poincar\'e inequality implies that we can choose as inner product
$$(u,h)=\int_{\Omega}(Du, Dh)_{\RR^{\NN}}dz\ \mbox{for all}\ u,h\in H^1_0(\Omega).$$

The corresponding norm is
$$||u||_{H^1_0(\Omega)}=||Du||_2\ \mbox{for all}\ u\in H^1_0(\Omega).$$

The space $C^1_0(\overline{\Omega})$ is an ordered Banach space with positive (order) cone given by
$$C_+=\{u\in C^1_0(\overline{\Omega}): u(z)\geq0\ \mbox{for all}\ z\in\overline{\Omega}\}.$$

This cone has a nonempty interior given by
$${\rm int}\, C_+ = \left\{u\in C_+: u(z)>0\ \mbox{for all}\ z\in\Omega, \left.\frac{\partial u}{\partial n}\right|_{\partial\Omega}<0\right\}.$$

Here, by $\frac{\partial u}{\partial n}$ we denote the normal derivative of $u$, with $n(\cdot)$ being the outward unit normal on $\partial\Omega$. Recall that $C^1_0(\overline{\Omega})$ is dense in both $W^{1,p}_0(\Omega)$ and in $H^1_0(\Omega)$.

We consider a function $f_0:\Omega\times\RR\rightarrow\RR$ which is Carath\'eodory, that is, for all $x\in\RR$ the mapping $z\mapsto f_0(z,x)$ is measurable and for almost all $z\in\Omega$ the function $x\mapsto f_0(z,x)$ is continuous. We assume that
$$|f_0(z,x)|\leq a_0(z)[1+|x|^{r-1}]\ \mbox{for almost all}\ z\in\Omega,\ \mbox{for all}\ x\in\RR,$$
with $a_0\in L^{\infty}(\Omega)$ and $1<r<p^{*}=\left\{\begin{array}{ll}
	\frac{Np}{N-p}\ & \mbox{if}\ p<N\\
	+\infty\ & \mbox{if}\ N\leq p
\end{array}\right.$ (the critical Sobolev exponent for $p$). We set $F_0(z,x) = \int^{x}_0 f_0(z,s)ds$ and consider the $C^{1}$-functional $\varphi_0: W^{1,p}_0(\Omega)\rightarrow\RR$ defined by
$$\varphi_0(u)=\frac{1}{p}||Du||^p_p + \frac{1}{2}||Du||^2_2 + \int_{\Omega} F_0(z,u)dz\quad \mbox{for all}\ u\in W^{1,p}_0(\Omega).$$

The next proposition is a special case of a more general result of Aizicovici, Papageorgiou \& Staicu \cite[Proposition 2]{2}. See also Papageorgiou \& R\u{a}dulescu \cite{23,24} for corresponding results for the Neumann and Robin problems. The result is essentially a byproduct of the regularity theory of Lieberman \cite[Theorem 1]{18}.

\begin{prop}\label{prop2}
	Assume that $u_0\in W^{1,p}_0(\Omega)$ is a local $C^{1}_0(\overline{\Omega})$-minimizer of $\varphi_0$, that is, there exists $\rho_0>0$ such that
	$$\varphi_0(u_0)\leq\varphi_0(u_0+h)\ \mbox{for all}\ h\in C^1_0(\overline{\Omega})\ \mbox{with}\ ||h||_{C^1_0(\overline{\Omega})}\leq\rho_0.$$
	Then $u_0\in C^{1,\alpha}_0(\overline{\Omega})$ for some $\alpha\in(0,1)$ and $u_0$ is a local $W^{1,p}_0(\Omega)$-minimizer of $\varphi_0$, that is, there exists $\rho_1>0$ such that
	$$\varphi_0(u_0)\leq\varphi_0(u_0+h)\ \mbox{for all}\ h\in W^{1,p}_0(\Omega)\ \mbox{with}\ ||h||\leq\rho_1.$$
\end{prop}

For any $r\in(1,+\infty)$ let $A_r:W^{1,r}_0(\Omega)\rightarrow W^{-1,r'}(\Omega)=W^{1,r}_0(\Omega)^{*}\left(\frac{1}{r}+\frac{1}{r'}=1\right)$ be the map defined by
$$\langle A_r(u), h\rangle = \int_{\Omega}|Du|^{r-2}(Du,Dh)_{\RR^N}dz\ \mbox{for all}\ u,h\in W^{1,r}_0(\Omega).$$

From Motreanu, Motreanu and Papageorgiou \cite[Proposition 2.72, p. 40]{19} we have the following property.

\begin{prop}\label{prop3}
	The map $A_r$ is bounded (that is, maps bounded sets to bounded sets), continuous, strictly monotone (hence, maximal monotone, too) and of type $(S)_+$, that is,
	$$``u_n\xrightarrow{w}u\ \mbox{in}\ W^{1,r}_0(\Omega)\ \mbox{and}\ \limsup_{n\rightarrow\infty}\langle A_r(u_n), u_n-u\rangle \leq 0\Rightarrow u_n\rightarrow u\ \mbox{in}\ W^{1,p}_0(\Omega)."$$
\end{prop}

Note that if $p=2$, then $A_2=A\in\mathscr{L}(H^1_0(\Omega), H^{-1}(\Omega)).$

We will use the spectrum of $(-\Delta_p, W^{1,p}_0(\Omega))$ and the Fu\v{c}ik spectrum of $(-\Delta, H^1_0(\Omega)).$ So, let us recall some basic facts about them.

We start with the following nonlinear eigenvalue problem
\begin{equation}\label{eq2}
	-\Delta_r u(z)=\hat{\lambda}|u(z)|^{r-2} u(z)\ \mbox{in}\ \Omega,\ u|_{\partial\Omega}=0,\ 1<r<\infty.
\end{equation}

We say that $\hat{\lambda}\in\RR$ is an ``eigenvalue" of $(-\Delta_r,W^{1,r}_0(\Omega))$, if problem (\ref{eq2}) admits a nontrivial solution $\hat{u}\in W^{1,r}_0(\Omega)$, known as an ``eigenfunction" corresponding to $\hat{\lambda}$. There is a smallest eigenvalue $\hat{\lambda}_1(r)>0$ such that:
\begin{itemize}
	\item $\hat{\lambda}_1(r)$ is isolated in the spectrum $\hat{\sigma}(r)$ of $(-\Delta_r,W^{1,r}_0(\Omega))$ (that is, there exists $\epsilon>0$ such that $(\hat{\lambda}_1(r),\hat{\lambda}_1(r)+\epsilon)\cap\hat{\sigma}(r)=\emptyset$).
	\item $\hat{\lambda}_1(r)$ is simple (that is, if $\hat{u}, \tilde{u}\in W^{1,r}_0(\Omega)$ are eigenfunctions corresponding to $\hat{\lambda}_1(r)$, then there exists $\xi\in\RR\backslash\{0\}$ such that $\hat{u}=\xi\tilde{u}$).
	\begin{equation}\label{eq3}
		\bullet\ \ \hat{\lambda}_1(r)=\inf\left[\frac{||Du||^r_r}{||u||^r_r}:\ u\in W^{1,r}_0(\Omega),\ u\neq0\right].\hspace{6.2cm}
	\end{equation}
\end{itemize}

In (\ref{eq3}) the infimum is realized on the one-dimensional eigenspace corresponding to $\hat{\lambda}_1(r)$. The  aforementioned properties imply that the elements of this eigenspace have fixed sign. Moreover, the nonlinear regularity theory (see, for example, Gasinski \& Papageorgiou \cite[pp. 737-738]{12}), implies that all the eigenfunctions of $(-\Delta_r, W^{1,r}_0(\Omega))$ belong in $C^{1}_0(\overline{\Omega})$. By $\hat{u}_1(r)$ we denote the positive $L^{r}$-normalized (that is, $||\hat{u}_1(r)||_r=1$) eigenfunction corresponding to $\hat{\lambda}_1(r)>0$. The nonlinear strong maximum principle (see, for example, Gasinski \& Papageorgiou \cite[p. 738]{12}) implies that $\hat{u}_1(r)\in {\rm int}\, C_+$. An eigenfunction $\hat{u}\in C^1_0(\overline{\Omega})$ corresponding to an eigenvalue $\hat{\lambda}\neq\hat{\lambda}_1(r)$ is necessarily nodal (sign-changing). It is easily seen that the set $\hat{\sigma}(r)$ is closed. Since $\hat{\lambda}_1(r)>0$ is isolated, the second eigenvalue $\hat{\lambda}_2(r)>0$ is well-defined by
$$\hat{\lambda}_2(r)=\min\left[\hat{\lambda}\in\hat{\sigma}(r): \hat{\lambda}\neq\hat{\lambda}_1(r)\right].$$

To produce additional eigenvalues, we can use the Ljusternik-Schnirelmann minimax scheme. In this way we obtain a whole nondecreasing sequence of eigenvalues $\{\hat{\lambda}_k(r)\}_{k\geq1}$ of $(-\Delta_r,W^{1,r}_0(\Omega))$ such that $\hat{\lambda}_k(r)\rightarrow+\infty$ as $k\rightarrow\infty$. These eigenvalues are known as ``variational eigenvalues" and $\hat{\lambda}_1(r), \hat{\lambda}_2(r)$ are as described above. We do not know if the variational eigenvalues exhaust the spectrum of $(-\Delta_r, W^{1,r}_0(\Omega))$. This is the case if $r=2$ (linear eigenvalue problem) or if $N=1$ (ordinary differential equations). In the linear case $(r=2)$, the eigenspaces $E(\hat{\lambda}_k(2)),\ k\in\NN$, are finite dimensional subspaces of $C^1_0(\overline\Omega)$ and we have the following orthogonal direct sum decomposition
$$H^1_0(\Omega)=\overline{\bigoplus_{k\geq1}E(\hat{\lambda}_k(2))}.$$

When $r\neq2$ (nonlinear eigenvalue problem), the eigenspaces are only cones and there is no decomposition of the space $W^{1,r}_0(\Omega)$ in terms of them. This makes the study of problems driven by $-\Delta_r$ and resonant at higher parts of the spectrum, difficult to deal with.

We will also encounter a weighted version of the eigenvalue problem (\ref{eq2}). So, let $m\in L^{\infty}(\Omega),\ m(z)\geq0$ for almost all $z\in\Omega,\ m\not\equiv0$. We consider the following nonlinear eigenvalue problem
\begin{equation}\label{eq4}
	-\Delta_r u(z)=\tilde{\lambda}m(z)|u(z)|^{r-2}u(z)\ \mbox{in}\ \Omega,\ u|_{\partial\Omega}=0.
\end{equation}

Again $\tilde{\lambda}\in\RR$ is an eigenvalue of $(-\Delta_r, W^{1,r}_0(\Omega),m)$, if problem (\ref{eq4}) admits a nontrivial solution. We have a smallest eigenvalue $\tilde{\lambda}_1(r,m)>0$ which is isolated, simple and
\begin{equation}\label{eq5}
	\tilde{\lambda}_1(r,m) = \inf\left[\frac{||Du||^r_r}{\int_\Omega m|u|^r dz}: u\in W^{1,r}_0(\Omega), u\neq0\right].
\end{equation}

As before, the infimum is realized on the corresponding one dimensional eigenspace, the elements of which do not change sign. This fact and (\ref{eq5}) lead to the following monotonicity property of $m\rightarrow\tilde{\lambda}_1(r,m)$.
\begin{prop}\label{prop4}
	If $m,m'\in L^{\infty}(\Omega)\backslash\{0\},\ 0\leq m(z)\leq m'(z)$ for almost all $z\in\Omega$ and $m\not\equiv m',$ then $\tilde{\lambda}_1(r,m')<\tilde{\lambda}_1(r,m)$
\end{prop}
\begin{remark}
	For the linear eigenvalue problem (that is $r=2$), the spectrum consists of a sequence $\{\tilde{\lambda}_k(2,m)=\tilde{\lambda}_k(m)\}_{k\in\NN}$ of distinct eigenvalues such that
	$$\tilde{\lambda}_k(m)\rightarrow+\infty\ \mbox{as}\ k\rightarrow\infty.$$
\end{remark}

The eigenspaces $E(\tilde{\lambda}_k(2,m))$ have the unique continuation property, that is, if $u\in E(\tilde{\lambda}_k(2,m))$ and $u(\cdot)$ vanishes on a set of positive Lebesgue measure, then $u\equiv0$. This property leads to the following strict monotonicity property of $\tilde{\lambda}_k(2,\cdot)$:
\begin{gather}
	m,m'\in L^{\infty}(\Omega)\backslash\{0\},\ 0\leq m(z)\leq m'(z)\ \mbox{for almost all}\ z\in\Omega, m\not\equiv m' \Rightarrow \nonumber \\
	\tilde{\lambda}_k(2,m')<\tilde{\lambda}_k(2,m)\ \mbox{for all}\ k\in\NN .\nonumber
\end{gather}

Another related result is the following lemma, which is a consequence of the properties of $\hat{\lambda}_1(p)$ (see Motreanu, Motreanu \& Papageorgiou \cite[Lemma 11.3, p. 305]{19}).
\begin{lemma}\label{lemma5}
	If $\vartheta\in L^{\infty}(\Omega)$ and $\vartheta(z)\leq\hat{\lambda}_1(p)$ for almost all $z\in\Omega,\ \vartheta\not\equiv\hat{\lambda}_1(p)$, then there exists $\hat{c}>0$ such that
	$$\hat{c}||u||^p\leq||Du||^p_p - \int_{\Omega}\vartheta(z)|u|^p dz\ \mbox{for all}\ u\in W^{1,p}_0(\Omega).$$
\end{lemma}

Since our problem is also asymmetric at zero, in our analysis, we will use the Fu\v{c}ik spectrum of $(-\Delta, H^1_0(\Omega))$. So, we consider the following linear eigenvalue problem
\begin{equation}\label{eq6}
	-\Delta u(z) = \alpha u^{+}(z) - \beta u^{-}(z)\ \mbox{in}\ \Omega,\ u|_{\partial\Omega}=0,
\end{equation}
where $u^{\pm}(\cdot)=\max\{\pm u(\cdot), 0\}$ (the positive and negative parts of $u$). By $\Sigma_2$ we denote the set of points $(\alpha, \beta)\in\RR^2$ for which problem (\ref{eq6}) admits a nontrivial solution. The set $\Sigma_2$ is called the ``Fu\v{c}ik spectrum" of $(-\Delta, H^1_0(\Omega))$. Let $\{\hat{\lambda}_k(2)\}_{k\in\NN}$ be the sequence of distinct eigenvalues of $(-\Delta, H^1_0(\Omega))$. While the spectrum of $(-\Delta, H^1_0(\Omega))$ is a sequence of points, the frame of the Fu\v{c}ik spectrum $\Sigma_2$ consists of a family of curves. In particular the lines $(\{\hat{\lambda}_1(2)\}\times\RR)\cup(\RR\times\{\hat{\lambda}_1(2)\})$ can be considered as the first curve of $\Sigma_2$. In fact this curve is isolated in $\Sigma_2$. For every $\ell\in\NN, \ell\geq2$, there are two decreasing curves $C_{\ell,1}, C_{\ell,2}$ (which may coincide) which pass through the point $(\hat{\lambda}_{\ell}(2), \hat{\lambda}_{\ell}(2))$ such that all points in the square $Q_{\ell}=(\hat{\lambda}_{\ell-1}(2), \hat{\lambda}_{\ell+1}(2))^2$ which are either in the region $I_{\ell,1}$ below both curves or in the region $I_{\ell,2}$ above the curves, do not belong to $\Sigma_2$ (these are the regions of type $I$). The status of the points between the two curves (when they do not coincide) is unknown in general. However, when $\hat{\lambda}_{\ell}(2)$ is a simple eigenvalue, points between the two curves are not in $\Sigma_2$. We mention that $\Sigma_2\subseteq\RR^2$ is closed with respect to the diagonal (that is, $(\alpha, \beta)\in\Sigma_2$ if and only if $( \beta,\alpha)\in\Sigma_2$). Also, $(\lambda,\lambda)\in\Sigma_2$ if and only if $\lambda=\hat{\lambda}_n(2)$ for some $n\in\NN$. As we already mentioned the lines $\{\hat\lambda_1(2)\}\times\RR$ and $\RR\times \{\hat\lambda_1(2)\}$ are contained in $\Sigma_2$. In the scalar case (that is, $N=1$) we have a complete description of the Fu\v{c}ik spectrum. For more information about $\Sigma_2$, we refer to Schechter \cite{29}.

Next, we recall some basic definitions and facts from Morse theory (critical groups). So, as before, $X$ is a Banach space, $\varphi\in C^1(X,\RR)$ and $c\in\RR$. We introduce the following sets
\begin{gather}
	K_{\varphi} = \{u\in X:\varphi'(u)=0\}, \nonumber \\
	K^c_{\varphi} = \{u\in K_{\varphi}: \varphi(u)=c\}, \nonumber \\
	\varphi^c = \{u\in X: \varphi(u)\leq c\} .\nonumber
\end{gather}

Let $(Y_1,Y_2)$ be a topological pair such that $Y_2\subseteq Y_1\subseteq X$ and $k\in\NN_0$. By $H_k(Y_1,Y_2)$ we denote the $k$th relative singular homology group with integer coefficients for the pair $(Y_1,Y_2)$. Suppose that $u\in K^c_{\varphi}$ is isolated. The critical groups of $\varphi$ at $u$ are defined by
$$C_k(\varphi,u)= H_k(\varphi^c\cap U, \varphi^c\cap U\backslash\{u\})\ \mbox{for all}\ k\in\NN_0,$$
where $U$ is a neighbourhood of $u$ such that $K_\varphi\cap\varphi^c\cap U=\{u\}$. The excision property of singular homology, implies that the above definition of critical groups is independent of the choice of the isolating neighbourhood $U$.
Suppose that $\varphi$ satisfies the $C$-condition and that $\inf\varphi(K_{\varphi})>-\infty$.
Let $c<\inf\varphi(K_\varphi)$. The critical groups of $\varphi$ at infinity are defined by
$$C_k(\varphi,\infty)=H_k(X,\varphi^c)\ \mbox{for all}\ k\in\NN_0.$$

This definition is independent of the choice of the level $c<\inf\varphi(K_\varphi)$. Indeed, suppose that $c'<c<\inf\varphi(K_\varphi)$. From  Motreanu, Motreanu \& Papageorgiou \cite[Corollary 6.35, p. 115]{19}, we have
\begin{align*}
	& \varphi^{c'}\ \mbox{is a strong deformation retract of}\ \varphi^{c}, \\
	\Rightarrow & H_k(X,\varphi^c) = H_k(X,\varphi^{c'})\ \mbox{for all}\ k\in\NN_0\ \mbox{(see \cite[Corollary 6.15(a), p. 145]{19})}.
\end{align*}

So, indeed $C_k(\varphi,\infty)$ is independent of the choice of the level $c<\inf\varphi(K_\varphi)$.

Now suppose that $\varphi\in C^1(X,\RR)$ satisfies the $C$-condition and that $K_\varphi$ is finite. We define
\begin{eqnarray*}
	&&M(t,u) = \sum_{k\in\NN_0}{\rm rank}\, C_k(\varphi,u)t^k\ \mbox{for all}\ t\in\RR, \mbox{all}\ u\in K_\varphi, \\
	&&P(t,\infty) = \sum_{k\in\NN_0}{\rm rank}\, C_k(\varphi, \infty)t^k\ \mbox{for all}\ t\in\RR.
\end{eqnarray*}

The Morse relation says that
\begin{equation}\label{eq7}
	\sum_{u\in K_\varphi} M(t,u) = P(t,\infty) + (1+t)Q(t)\ \mbox{for all}\ t\in\RR,
\end{equation}
where $Q(t)=\sum_{k\in\NN_0}\beta_k t^k$ is a formal series in $t\in\RR$, with nonnegative integer coefficients $\beta_k$.

We conclude this section by fixing our notation and introducing the hypotheses on the reaction term $f(z,x)$. Recall that if $u\in W^{1,p}_0(\Omega)$, we define
$$u^{\pm}(z)=\max\{\pm u(z),0\}.$$

We know that $u^{\pm}\in W^{1,p}_{0}(\Omega)$ and we have $u=u^+-u^-$, $|u|=u^++u^-$. By $|\cdot|_N$ we denote the Lebesgue measure on $\RR^N$ and given $f(z,x)$ a measurable function (for example, a Carath\'eodory function), we denote by $N_f(\cdot)$  Nemitsky (superposition) map corresponding to $f(\cdot,\cdot)$ and defined by
$$N_f(u)(\cdot)=f(\cdot,u(\cdot))\ \mbox{for all}\ u\in W^{1,p}_0(\Omega).$$

The hypotheses on $f(z,x)$ are the following:

\smallskip
$H(f):f:\Omega\times\RR\rightarrow\RR$ is a Carath\'eodory function such that $f(z,0)=0$ for almost all $z\in\Omega$ and
\begin{itemize}
	\item [(i)] for every $r>0$, there exists $a_r\in L^{\infty}(\Omega)_+$ such that
		$$|f(z,x)|\leq a_r(z)\ \mbox{for almost all}\ z\in\Omega,\ \mbox{for all}\ |x|\leq r;$$
	\item [(ii)] there exists $\eta\in L^{\infty}(\Omega), \eta(z)\geq\hat{\lambda}_1(p)$ for almost all $z\in\Omega$, $\eta\neq\hat{\lambda}_1(p)$ and $\hat{\eta},\hat{\vartheta}>0$ such that
		\begin{gather*}
			-\hat{\vartheta}\leq\liminf_{x\rightarrow-\infty}\frac{f(z,x)}{|x|^{p-2}x} \leq \limsup_{x\rightarrow-\infty}\frac{f(z,x)}{|x|^{p-2}x} \leq \hat{\lambda}_1(p) \\
			\eta(z) \leq \liminf_{x\rightarrow+\infty}\frac{f(z,x)}{|x|^{p-2}x} \leq \limsup_{x\rightarrow+\infty}\frac{f(z,x)}{|x|^{p-2}x} \leq \hat{\eta}
		\end{gather*}
		uniformly for almost all $z\in\Omega$;
	\item [(iii)] if $F(z,x)=\int^x_0 f(z,s)ds$, then $f(z,x)x - pF(z,x)\rightarrow+\infty$ as $x\rightarrow-\infty$ uniformly for almost all $z\in\Omega$ and there exists $M_0>0$ such that
		$$f(z,x)x - pF(z,x)\geq0\ \mbox{for almost all}\ z\in\Omega,\ \mbox{for all}\ x\geq M_0;$$
	\item [(iv)] there exists $0<\alpha<\hat{\lambda}_1(2)<\beta<\hat{\lambda}_2(2)$ and
		$$\lim_{x\rightarrow0^+}\frac{f(z,x)}{x}=\alpha,\ \lim_{x\rightarrow0^-}\frac{f(z,x)}{x}=\beta$$
		uniformly for almost all $z\in\Omega$ and for every $\rho>0$, there exists $\hat{\xi}_\rho>0$ such that for almost all $z\in\Omega$ the mapping $x\mapsto f(z,x) + \hat{\xi}_\rho x^{p-1}$ is nondecreasing on $[0, \rho]$.
\end{itemize}
\begin{remark}
	Hypothesis $H(f)(ii)$ implies that $f(z,\cdot)$ is a crossing nonlinearity. In fact we can cross any finite number of variational eigenvalues starting with $\hat{\lambda}_1(p)>0$. Note that in the negative direction we can have resonance with respect to $\hat{\lambda}_1(p)>0$, while in the positive direction resonance is possible with respect to any nonprincipal eigenvalue of $(-\Delta_p, W^{1,p}_0(\Omega))$. As we will see in the proof of Proposition \ref{prop8}, hypothesis $H(f)(iii)$ guarantees that at $-\infty$ the resonance with respect to $\hat{\lambda}_1(p)>0$, is from the left of the principal eigenvalue in the sense that
	$$\hat{\lambda}_1(p)|x|^p-pF(z,x)\rightarrow+\infty\ \mbox{as}\ x\rightarrow-\infty,\ \mbox{uniformly for almost all}\ z\in\Omega.$$

This makes the negative truncation of the energy functional of (\ref{eq1}) coercive. So, we can use the direct method of the calculus of variations. Hypothesis $H(f)(iv)$ implies that at zero, too, we have an asymmetric behaviour of the quotient $\frac{f(z,x)}{x}$.
\end{remark}

\section{Solutions of Constant sign}

In this section, using variational tools, we show that problem (\ref{eq1}) admits two nontrivial smooth solutions of constant sign (one positive and the other negative).

So, let $\varphi:W^{1,p}_0(\Omega)\rightarrow\RR$ be the energy functional for problem (\ref{eq1}) defined by
$$\varphi(u)=\frac{1}{p}||Du||^p_p + \frac{1}{2}||Du||^2_2 - \int_\Omega F(z,u)dz\ \mbox{for all}\ u\in W^{1,p}_0(\Omega).$$

Evidently, $\varphi\in C^1(W^{1,p}_0(\Omega),\RR)$. Also, we consider the positive and negative truncations of $f(z,\cdot)$, that is, the Carath\'eodory function
$$f_\pm(z,x)=f(z,\pm x^\pm).$$

We set $F_\pm(z,x)=\int^x_0 f_\pm(z,s)ds$ and consider the $C^1$-functionals $\varphi_\pm:W^{1,p}_0(\Omega)\rightarrow\RR$ defined by
$$\varphi_\pm(u)=\frac{1}{p}||Du||^p_p+\frac{1}{2}||Du||^2_2 - \int_\Omega F_\pm(z,u)dz\ \mbox{for all}\ u\in W^{1,p}_0(\Omega).$$

\begin{prop}\label{prop6}
	If hypotheses $H(f)$ hold, then $\varphi$ satisfies the $C$-condition.
\end{prop}
\begin{proof}
	We consider a sequence $\{u_n\}_{n\geq1}\leq W^{1,p}_0(\Omega)$ such that
	\begin{equation}\label{eq8}
		|\varphi(u_n)|\leq M_1\ \mbox{for some}\ M_1>0,\ \mbox{all}\ n\in\NN
	\end{equation}
	\begin{equation}\label{eq9}
		(1+||u_n||)\varphi'(u_n)\rightarrow0\ \mbox{in}\ W^{-1,p'}(\Omega) = W^{1,p}_0(\Omega)^{*}\ \mbox{as}\ n\rightarrow\infty.
	\end{equation}
	
	From (\ref{eq9}) we have
	\begin{eqnarray}\label{eq10}
		&&\left|\langle A_p(u_n),h\rangle + \langle A(u_n),h\rangle - \int_\Omega f(z,u_n)hdz\right| \leq \frac{\epsilon_n||h||}{1+||u_n||}, \\
		&&\mbox{for all}\ h\in W^{1,p}_0(\Omega), \mbox{with}\ \epsilon_n\rightarrow0^+.\nonumber
	\end{eqnarray}
	
	In (\ref{eq10}) we chose $h=u_n\in W^{1,p}_0(\Omega)$. Then
	\begin{equation}\label{eq11}
		-||Du_n||^p_p - ||Du_n||^2_2 + \int_\Omega f(z,u_n)u_n dz\leq\xi_n\ \mbox{for all}\ n\in\NN.
	\end{equation}
	
	Also, from (\ref{eq8}) we have
	\begin{equation}\label{eq12}
		||Du_n||^p_p+ \frac{p}{2}||Du_n||^2_2 - \int_\Omega pF(z,u_n)dz\leq pM_1\ \mbox{for all}\ n\in\NN.
	\end{equation}
	
	We add (\ref{eq11}) and (\ref{eq12}). Recalling that $p>2$, we obtain
	\begin{eqnarray}\label{eq13}
		&& \int_\Omega[f(z,u_n)u_n - pF(z,u_n)]dz\leq M_2\ \mbox{for some}\ M_2>0, \mbox{all}\ n\in\NN, \nonumber \\
		&\Rightarrow & \int_\Omega[f(z,-u^-_n)(-u^-_n)-pF(z,-u^-_n)]dz \leq M_3\ \mbox{for some}\ M_3>0\ \mbox{all}\ n\in\NN \\
		&& \mbox{(see hypotheses $H(f)(i)(iii)$)}. \nonumber
	\end{eqnarray}
	\begin{claim}\label{claim1}
		$\{u^-_n\}_{n\geq1}\subseteq W^{1,p}_0(\Omega)$ is bounded.
	\end{claim}
	We argue by contradiction. So, suppose that Claim \ref{claim1} is not true. By passing to a subsequence if necessary, we may assume that
	\begin{equation}\label{eq14}
		||u^-_n||\rightarrow\infty\ \mbox{as}\ n\rightarrow\infty.
	\end{equation}
	
	Let $y_n=\frac{u^-_n}{||u^-_n||},\ n\in\NN$. Then $||y_n||=1,\ y_n\geq0$ for all $n\in\NN$. So, we may assume that
	\begin{equation}\label{eq15}
		y_n\xrightarrow{w}y\ \mbox{in}\ W^{1,p}_0(\Omega)\ \mbox{and}\ y_n\rightarrow y\ \mbox{in}\ L^p(\Omega),\ y\geq0.
	\end{equation}
	
	In (\ref{eq10}) we choose $h=-u^-_n\in W^{1,p}_0(\Omega)$. Then
	\begin{eqnarray}\label{eq16}
		&& ||Du^-_n||^p_p + ||Du^-_n||^2_2 - \int_\Omega f(z,-u^-_n)(-u^-_n)dz\leq\xi_n\ \mbox{for all}\ n\in\NN, \nonumber \\
		&\Rightarrow & ||Dy_n||^p_p + \frac{1}{||u^-_n||^{p-2}}||Dy_n||^2_2 - \int_\Omega\frac{N_f(-u^-_n)}{||u^-_n||^{p-1}}y_n dz \leq \frac{\xi_n}{||u^-_n||^{p-1}}\ \mbox{for all}\ n\in\NN.
	\end{eqnarray}
	
	Hypotheses $H(f)(i),(ii),(iii)$ imply that
	\begin{equation}\label{eq17}
		|f(z,x)|\leq c_1[1+|x|^{p-1}]\ \mbox{for almost all}\ z\in\Omega,\ \mbox{all}\ x\in\RR,\ \mbox{for some}\ c_1>0.
	\end{equation}
	
	From (\ref{eq17}) it follows that
	\begin{equation}\label{eq18}
		\left\{\frac{N_f(-u^-_n)}{||u^-_n||^{p-1}}\right\}_{n\geq1} \subseteq L^{p'}(\Omega)\ \mbox{is bounded}\ \left(\frac{1}{p}+\frac{1}{p'}=1\right).
	\end{equation}
	
	On account of (\ref{eq18}) and hypothesis $H(f)(ii)$, at least for a subsequence we have
	\begin{equation}\label{eq19}
		\frac{N_f(-u^-_n)}{||u^-_n||^{p-1}}\xrightarrow{w}\vartheta(z)y^{p-1}\ \mbox{in}\ L^{p'}(\Omega)\ \mbox{with}\ -\hat{\vartheta}\leq\vartheta(z)\leq\hat{\lambda}_1(p)\ \mbox{for almost all}\ z\in\Omega
	\end{equation}
	(see Aizicovici, Papageorgiou \& Staicu \cite{1}, proof of Proposition 16). We pass to the limit as $n\rightarrow\infty$ in (\ref{eq16}). Using (\ref{eq14}), (\ref{eq15}), (\ref{eq19}) and the fact that $2<p$ we obtain
	\begin{equation}\label{eq20}
		||Dy||^p_p\leq\int_\Omega\vartheta(z)y^pdz.
	\end{equation}
	
	If $\vartheta\not\equiv\hat{\lambda}_1(p)$, then from (\ref{eq14}) and Lemma \ref{lemma5} we have
	\begin{eqnarray}\label{eq21}
		& \hat{c}||y||^p\leq0, \nonumber \\
		\Rightarrow & y=0.
	\end{eqnarray}
	
	Then from (\ref{eq16}) and using as before (\ref{eq14}), (\ref{eq15}), (\ref{eq19}) (the last two relations with $y=0$, see (\ref{eq21})) and the fact that $p>2$, we infer that
	\begin{eqnarray*}
		&& ||Dy_n||_p\rightarrow0, \\
		&\Rightarrow & y_n\rightarrow0\ \mbox{in}\ W^{1,p}_0(\Omega),
	\end{eqnarray*}
	which contradicts the fact that $||y_n||=1$ for all $n\in\NN$.
	
	Next we assume that $\vartheta(z)=\hat{\lambda}_1(p)$ for almost all $z\in\Omega$ (resonant case). Then from (\ref{eq20}) and (\ref{eq3}) we have
	\begin{eqnarray*}
		&& ||Dy||^p_p=\hat{\lambda}_1(p)||y||^p_p, \\
		&\Rightarrow & y=\tilde{\vartheta}\hat{u}_1(p)\ \mbox{with}\ \tilde{\vartheta}\geq0\ \mbox{(recall that $y\geq0$, see (\ref{eq15}))}.
	\end{eqnarray*}
	
	If $\tilde{\vartheta}=0$, then $y=0$ and as above we have
	$$y_n\rightarrow0\ \mbox{in}\ W^{1,p}_0(\Omega),$$
	again contradicting the fact that $||y_n||=1$ for all $n\in\NN$.
	
	If $\tilde{\vartheta}>0$, then $y(z)>0$ for all $z\in\Omega$ and so
	\begin{eqnarray}\label{eq22}
		&& u^-_n(z)\rightarrow+\infty\ \mbox{for all}\ z\in\Omega, \nonumber \\
		&\Rightarrow & f_n(z,-u^-_n(z))(-u^-_n(z)) - pF(z,-u^-_n(z))\rightarrow0\ \mbox{for almost all}\ z\in\Omega,\ \mbox{as}\ n\rightarrow\infty \nonumber \\
		&& \mbox{(see hypothesis $H(f)(iii)$)} \nonumber \\
		&\Rightarrow & \int_\Omega\left[f(z,-u^-_n)(-u^-_n) - pF(z,-u^-_n)\right]dz\rightarrow+\infty\quad  \mbox{(by Fatou's lemma)}. 	\end{eqnarray}
	
	We compare (\ref{eq22}) and (\ref{eq13}) and have a contradiction.
	
	This proves Claim \ref{claim1}.
	\begin{claim}\label{claim2}
		$\{u^+_n\}_{n\geq1}\subseteq W^{1,p}_0(\Omega)$ is bounded.
	\end{claim}
	Again we argue indirectly. So, suppose that Claim \ref{claim2} is not true. Then at least for a subsequence we have
	\begin{equation}\label{eq23}
		||u^+_n||\rightarrow+\infty\ \mbox{as}\ n\rightarrow\infty.
	\end{equation}
	
	Let $v_n=\frac{u^+_n}{||u^+_n||},\ n\in\NN$. Then $||v_n||=1,\ v_n\geq0$ for all $n\in\NN$ and so we may assume that
	\begin{equation}\label{eq24}
		v_n\xrightarrow{w}v\ \mbox{in}\ W^{1,p}_0(\Omega)\ \mbox{and}\ v_n\rightarrow v\ \mbox{in}\ L^p(\Omega),\ v\geq0.
	\end{equation}
	
	From (\ref{eq10}) and Claim \ref{claim1}, we have
	\begin{eqnarray}\label{eq25}
		&& |\langle A_p(u^+_n),h\rangle + \langle A(u^+_n),h\rangle - \int_\Omega f(z,u^+_n)hdz|\leq M_4 \nonumber \\
		&& \mbox{for some $M_4>0$, all $n\in\NN$, all $h\in W^{1,p}_0(\Omega)$ (see hypothesis $H(f)(i)$)} \nonumber \\
		&\Rightarrow & \left|\langle A_p(v_n),h\rangle + \frac{1}{||u^+_n||^{p-2}}, \langle A(v_n),h\rangle - \int_\Omega\frac{N_f(u^+_n)}{||u^+_n||^{p-1}}hdz\right| \leq \frac{M_3||h||}{||u^+_n||^{p-1}}\ \mbox{for all}\ n\in\NN.
	\end{eqnarray}
	
	Using the growth condition from (\ref{eq17}), we see that
	\begin{equation}\label{eq26}
		\left\{\frac{N_f(u^+_n)}{||u^+_n||^{p-1}}\right\}_{n\geq1} \subseteq L^{p'}(\Omega)\ \mbox{is bounded}.
	\end{equation}
	
	In (\ref{eq25}) we choose $h=v_n-v\in W^{1,p}_0(\Omega)$, pass to the limit as $n\rightarrow\infty$ and use (\ref{eq23}), (\ref{eq24}), (\ref{eq20}) and the fact that $p>2$. Then
	\begin{eqnarray}\label{eq27}
		&& \lim_{n\rightarrow\infty}\langle A_p(v_n),v_n-v\rangle=0, \nonumber \\
		&\Rightarrow & v_n\rightarrow v\ \mbox{in}\ W^{1,p}_0(\Omega)\ \mbox{(see Proposition \ref{prop3})}.
	\end{eqnarray}
	
	From (\ref{eq26}) and hypothesis $H(f)(ii)$, we see that at least for a subsequence we have
	\begin{equation}\label{eq28}
		\frac{N_f(u^+_n)}{||u^+_n||^{p-1}}\xrightarrow{w}\tilde{\eta}(z)v^{p-1}\ \mbox{in}\ L^{p'}(\Omega)\ \mbox{with}\ \eta(z)\leq\hat{\eta}(z)\leq\hat{\eta}\ \mbox{for almost all}\ z\in\Omega\ \mbox{(see \cite{1})}.
	\end{equation}
	
	So, if in (\ref{eq25}) we pass to the limit as $n\rightarrow\infty$ and use (\ref{eq23}), (\ref{eq27}), (\ref{eq28}) and the fact that $p>2$, then
	\begin{eqnarray}\label{eq29}
		& \langle A_p(v),h\rangle = \int_\Omega\tilde{\eta}(z)v^{p-1}hdz\ \mbox{for all}\ h\in W^{1,p}_0(\Omega), \nonumber \\
		\Rightarrow & -\Delta_p v(z) = \tilde{\eta}(z) v(z)^{p-1}\ \mbox{for almost all}\ z\in\Omega,\ v|_{\partial\Omega}=0.
	\end{eqnarray}
	
	From Proposition \ref{prop4}, we have
	\begin{equation}\label{eq30}
		\tilde{\lambda}_1(p,\tilde{\eta}) < \tilde{\lambda}_1(p,\hat{\lambda}_1)=1.
	\end{equation}
	
	From (\ref{eq29}) and (\ref{eq30}) and since $||v||=1$ (see (\ref{eq27})), it follows that $v(\cdot)$ must be  nodal, contradicting (\ref{eq24}).
	This proves Claim \ref{claim2}.
	
	From claims \ref{claim1} and  \ref{claim2}, we deduce that
	$$\{u_n\}_{n\geq1}\subseteq W^{1,p}_0(\Omega)\ \mbox{is bounded}.$$
	
	So, we may assume that
	\begin{equation}\label{eq31}
		u_n\xrightarrow{w}u\ \mbox{in}\ W^{1,p}_0(\Omega)\ \mbox{and}\ u_n\rightarrow u\ \mbox{in}\ L^p(\Omega).
	\end{equation}
	
	In (\ref{eq10}) we choose $h=u_n-u\in W^{1,p}_0(\Omega)$, pass to the limit as $n\rightarrow\infty$ and use (\ref{eq31}) and the fact that $\{N	_f(u_n)\}_{n\geq1}\subseteq L^{p'}(\Omega)$ is bounded (see (\ref{eq17})). Then
	\begin{align*}
		& \lim_{n\rightarrow\infty}\left[\langle A_p(u_n), u_n-u\rangle + \langle A(u_n),u_n-u\rangle\right]=0, \\
		\Rightarrow & \limsup_{n\rightarrow\infty}\left[\langle A_p(u_n),u_n-u\rangle + \langle A(u), u_n-u\rangle\right]\leq0\ \mbox{(since $A(\cdot)$ is monotone)}, \\
		\Rightarrow & u_n\rightarrow u\ \mbox{in}\ W^{1,p}_0(\Omega)\ \mbox{(see Proposition \ref{prop3})}.
	\end{align*}
	
	Therefore the energy functional $\varphi$ satisfies the $C$-condition.
\end{proof}

Next, we show that $\varphi_+$ satisfies the $C$-condition, too.
\begin{prop}\label{prop7}
	If hypotheses $H(f)$ hold, then $\varphi_+$ satisfies the $C$-condition.
\end{prop}
\begin{proof}
	We consider a sequence $\{u_n\}_{n\geq1}\subseteq W^{1,p}_0(\Omega)$ such that
	\begin{eqnarray}
		|\varphi_+(u_n)|\leq M_5\ \mbox{for some}\ M_5>0,\ \mbox{for all}\ n\in\NN, \label{eq32} \\
		(1+||u_n||)\varphi'_+(u_n)\rightarrow0\ \mbox{in}\ W^{-1,p'}(\Omega)\ \mbox{as}\ n\rightarrow\infty . \label{eq33}
	\end{eqnarray}
	
	From (\ref{eq33}) we have
	\begin{eqnarray}\label{eq34}
		&&|\langle A_p(u_n),h\rangle + \langle A(u_n),h\rangle - \int_\Omega f_+(z,u_n)hdz| \leq \frac{\epsilon_n||h||}{1+||u_n||} \\
		&&\mbox{for all}\ h\in W^{1,p}_0(\Omega)\ \mbox{with}\ \epsilon_n\rightarrow0^+.\nonumber
	\end{eqnarray}
	
	In (\ref{eq34}) we choose $h=-u^-_n\in W^{1,p}_0(\Omega)$. Then
	\begin{eqnarray}\label{eq35}
		& ||Du^-_n||^p_p + ||Du^-_n||^2_2\leq\epsilon_n\ \mbox{for all}\ n\in\NN, \nonumber \\
		\Rightarrow & u^-_n\rightarrow0\ \mbox{in}\ W^{1,p}_0(\Omega)\ \mbox{as}\ n\rightarrow\infty.
	\end{eqnarray}
	
	From (\ref{eq34}) and (\ref{eq35}) it follows that
	\begin{equation}\label{eq36}
		|\langle A_p(u^+_n),h\rangle +\langle A(u^+_n),h\rangle - \int_\Omega f(z,u^+_n)hdz| \leq \epsilon'_n||h||\ \mbox{for all}\ h\in W^{1,p}_0(\Omega)\ \mbox{with}\ \epsilon'_n\rightarrow0.
	\end{equation}
	
	Suppose that $\{u^+_n\}_{n\geq1}\subseteq W^{1,p}_0(\Omega)$ is unbounded. So, we may assume that $||u^+_n||\rightarrow\infty$. We set $v_n=\frac{u^+_n}{||u^+_n||},\ n\in\NN$ and have $||v_n||=1,\ v_n\geq0$ for all $n\in\NN$. Hence we can say (at least for a subsequence) that
	$$v_n\xrightarrow{w}v\ \mbox{in}\ W^{1,p}_0(\Omega)\ \mbox{and}\ v_n\rightarrow v\ \mbox{in}\ L^p(\Omega).$$
	
	Then reasoning as in the proof of Proposition \ref{prop6} (see the part of the proof after (\ref{eq24})), we show that
	\begin{equation}\label{eq37}
		\{u^+_n\}_{n\geq1}\subseteq W^{1,p}_0(\Omega)\ \mbox{is bounded}.
	\end{equation}
	
	From (\ref{eq35}) and (\ref{eq37}) it follows that
	$$\{u_n\}_{n\geq1}\subseteq W^{1,p}_0(\Omega)\ \mbox{is bounded}.$$
	
	So, we may assume that
	\begin{equation}\label{eq38}
		u_n\xrightarrow{w}u\ \mbox{in}\ W^{1,p}_0(\Omega)\ \mbox{and}\ u_n\rightarrow u\ \mbox{in}\ L^p(\Omega).
	\end{equation}
	
	In (\ref{eq34}) we choose $h=u_n-u\in W^{1,p}_0(\Omega)$. Passing the limit as $n\rightarrow\infty$, using (\ref{eq38}) and following the argument in the last part of the proof of Theorem \ref{th1} (see the part of the proof after (\ref{eq31})), we obtain
	$u_n\rightarrow u$ in $W^{1,p}_0(\Omega)$.  We conclude that
		$\varphi_+$ satisfies the $C$-condition.
\end{proof}

For the functional $\varphi_-$, we have the following result.
\begin{prop}\label{prop8}
	If hypotheses $H(f)$ hold, then $\varphi_-$ is coercive.
\end{prop}
\begin{proof}
	Hypothesis $H(f)(iii)$ implies that given any $\xi>0$, we can find $M_6=M_6(\xi)>0$ such that
	\begin{equation}\label{eq39}
		f(z,x)x - pF(z,x)\geq\xi\ \mbox{for almost all}\ z\in\Omega,\ \mbox{for all}\ x\leq-M_6.
	\end{equation}
	
	We have
	\begin{align}\label{eq40}
		\frac{d}{dx}\left[\frac{F(z,x)}{|x|^p}\right] & = \frac{f(z,x)|x|^p - p|x|^{p-2}xF(z,x)}{|x|^{2p}} \nonumber \\
		& = \frac{f(z,x)x - pF(z,x)}{|x|^px} \nonumber \\
		& \leq \frac{\xi}{|x|^px}\ \mbox{for almost all}\ z\in\Omega,\ \mbox{for all}\ x\leq M_6\ \mbox{(see (\ref{eq39}))}, \nonumber \\
		\Rightarrow \frac{F(z,y)}{|y|^p} - \frac{F(z,w)}{|w|^p} & \geq\frac{\xi}{p}\left[\frac{1}{|w|^p} - \frac{1}{|y|^p}\right]\ \mbox{for almost all}\ z\in\Omega, \mbox{all}\ y\leq w\leq-M_6.
	\end{align}
	
	Hypothesis $H(f)(iii)$ implies that
	\begin{equation}\label{eq41}
		-\hat{\vartheta}\leq\liminf_{x\rightarrow-\infty}\frac{pF(z,x)}{|x|^p} \leq \limsup_{x\rightarrow-\infty}\frac{pF(z,x)}{|x|^p} \leq \hat{\lambda}_1(p)\ \mbox{uniformly for almost all}\ z\in\Omega.
	\end{equation}
	
	If in (\ref{eq40}) we pass to the limit as $y\rightarrow-\infty$ and use (\ref{eq41}), then
	$$\hat{\lambda}_1(p)|w|^p - pF(z,w)\geq \xi\ \mbox{for almost all}\ z\in\Omega,\ \mbox{for all}\ w\leq -M_6.$$
	
	But $\xi>0$ is arbitrary. So, we infer that
	\begin{equation}\label{eq42}
		\hat{\lambda}_1(p)|w|^p - pF(z,w)\rightarrow+\infty\ \mbox{as}\ w\rightarrow-\infty, \mbox{uniformly for almost all}\ z\in\Omega.
	\end{equation}
	
	We will use (\ref{eq42}) to show that $\varphi_-$ is coercive. We argue by contradiction. So, suppose that $\varphi_-$ is not coercive. Then we can find $\{u_n\}_{n\geq1}\subseteq W^{1,p}_0(\Omega)$ and $M_7>0$ such that
	\begin{equation}\label{eq43}
		||u_n||\rightarrow\infty\ \mbox{and}\ \varphi_-(u_n)\leq M_7\ \mbox{for all}\ n\in\NN.
	\end{equation}
	
	Let $y_n=\frac{u_n}{||u_n||},\ n\in\NN$. Then $||y_n||=1$ for all $n\in\NN$ and so we may assume that
	\begin{equation}\label{eq44}
		y_n\xrightarrow{w}y\ \mbox{in}\ W^{1,p}_0(\Omega)\ \mbox{and}\ y_n\rightarrow y\ \mbox{in}\ L^p(\Omega).
	\end{equation}
	
	From (\ref{eq43}) we have
	\begin{equation}\label{eq45}
		||Dy_n||^p_p + \frac{p}{2||u_n||^{p-2}}||Dy_n||^2_2 - \int_\Omega\frac{pF_-(z,u_n)}{||u_n||^p}dz \leq \frac{M_7}{||u_n||^p}\ \mbox{for all}\ n\in\NN.
	\end{equation}
	
	From (\ref{eq17}) we have
	\begin{eqnarray*}
		|F(z,x)|\leq c_3[1+|x|^p]\ \mbox{for almost all}\ z\in\Omega, \mbox{all}\ x\in\RR, \mbox{some}\ c_3>0, \\
		\Rightarrow\left\{\frac{F_-(\cdot,u_n(\cdot))}{||u_n||^p}\right\}_{n\geq1}\subseteq L^1(\Omega)\ \mbox{is uniformly integrable}.
	\end{eqnarray*}
	
	By the Dunford-Pettis theorem and (\ref{eq41}) we have (at least for a subsequence) that
	\begin{equation}\label{eq46}
		\frac{F_-(\cdot,u_n(\cdot))}{||u_n||^p}\xrightarrow{w}\frac{1}{p}\vartheta(z)(y^-)^p\ \mbox{in}\ L^1(\Omega)\ \mbox{with}\ -\hat{\vartheta}\leq\vartheta(z)\leq\hat{\lambda}_1(p)\ \mbox{for almost all}\ z\in\Omega.
	\end{equation}
	
	We return to (\ref{eq45}), pass to the limit as $n\rightarrow\infty$ and use (\ref{eq43}), (\ref{eq44}), (\ref{eq46}) together with the fact that $p>2$. Then
	\begin{eqnarray}
		||Dy||^p_p\leq\int_{\Omega}\vartheta(z)(y^-)^pdz, \label{eq47} \\
		\Rightarrow ||Dy^-||^p_p\leq\int_{\Omega}\vartheta(z)(y^-)^pdz. \label{eq48}
	\end{eqnarray}
	
	If $\vartheta\not\equiv\hat{\lambda}_1(p)$, then from (\ref{eq48}) and Lemma \ref{lemma5}, we have
	\begin{align*}
		& \hat{c}||y^-||^p\leq0, \\
		\Rightarrow & y\geq0.
	\end{align*}
	
	Then (\ref{eq47}) implies that $y=0$. So, from (\ref{eq45}) it follows that
	$$y_n\rightarrow0\ \mbox{in}\ W^{1,p}_0(\Omega),$$
	which contradicts the fact that $||y_n||=1$ for all $n\in\NN$.
	
	Next, assume that $\vartheta(z)=\hat{\lambda}_1(p)$ for almost all $z\in\Omega$. From (\ref{eq48}) and (\ref{eq3}), we have
	\begin{align*}
		& ||Dy_-||^p_p = \hat{\lambda}_1(p)||y^-||^p_p, \\
		\Rightarrow & y^-=\tilde{\vartheta}\hat{u}_1(p)\ \mbox{with}\ \tilde{\vartheta}\geq0.
	\end{align*}
	
	If $\tilde{\vartheta}=0$, then $y\geq0$ and as above we reach a contradiction.
	
	If $\tilde{\vartheta}>0$, then $y(z)<0$ for all $z\in\Omega$ and so
	\begin{eqnarray*}
		&& u_n(z)\rightarrow-\infty\ \mbox{for all}\ z\in\Omega\ \mbox{as}\ n\rightarrow\infty, \\
		&\Rightarrow & u^-_n(z)\rightarrow+\infty\ \mbox{for all}\ z\in\Omega\ \mbox{as}\ n\rightarrow\infty, \\
		&\Rightarrow & \hat{\lambda}_1(p)u^-_n(z)^p - pF(z,-u^-_n(z))\rightarrow+\infty\ \mbox{for almost all}\ z\in\Omega\ \mbox{(see (\ref{eq42}))} \\
		&\Rightarrow & \int_\Omega[\hat{\lambda}_1(p)(u^-_n)^p - pF(z,-u^-_n)]dz\rightarrow+\infty\ \mbox{(by Fatou's lemma).}
	\end{eqnarray*}
	
	Since $\hat{\lambda}_1(p)||u^-_n||^p_p\leq||Du^-_n||^p_p$ for all $n\in\NN$ (see (\ref{eq3})), it follows that
	$$p\varphi_-(u_n)\rightarrow+\infty\ \mbox{as}\ n\rightarrow\infty,$$
	which contradicts (\ref{eq43}). Therefore $\varphi_-$ is coercive.
\end{proof}
\begin{remark}
	From (\ref{eq42}) we see that the resonance with respect to $\hat{\lambda}_1(p)>0$ at $-\infty$, is from the left of the principal eigenvalue.
\end{remark}
From the above proposition we infer the following fact about the functional $\varphi_-$ (see \cite{19}).
\begin{corollary}\label{corollary9}
	If hypotheses $H(f)$ hold, the $\varphi_-$ satisfies the $C$-condition.
\end{corollary}

Next, we determine the nature of the critical point $u=0$ for $\varphi_+$.
\begin{prop}\label{prop10}
	If hypotheses $H(f)$ hold, then $u=0$ is a local minimizer for $\varphi_+$.
\end{prop}
\begin{proof}
	Hypothesis $H(f)(iv)$ implies that given $\epsilon>0$, we can find $\delta=\delta(\epsilon)>0$ such that
	\begin{equation}\label{eq49}
		F(z,x)\leq\frac{1}{2}\,(\alpha+\epsilon)x^2\ \mbox{for almost all}\ z\in\Omega,\ \mbox{for all}\ 0\leq x\leq\delta.
	\end{equation}
	
	Let $u\in C^1_0(\overline\Omega)$ with $||u||_{C^1_0(\overline\Omega)}\leq\delta$. Then
	\begin{eqnarray*}
		\varphi_+(u) &\geq & \frac{1}{p}||Du||^p_p+\frac{1}{2}||Du||^2_2 - \frac{\alpha+\epsilon}{2}||u^+||^2_2\ \mbox{(see (\ref{eq49}))} \\
		&\geq & \frac{1}{p}||Du||^p_p + \frac{1}{2}||Du^-||^2_2 + \frac{1}{2}\left[||Du^+||^2_2 - \alpha||u^+||^2_2\right] - \frac{\epsilon}{2\hat{\lambda}_1(2)}||Du^+||^2_2\ \mbox{(see (\ref{eq3}))} \\
		&\geq & \frac{1}{p}||Du||^p_p + \frac{1}{2}||Du^-||^2_2+\frac{1}{2}\left[c_4-\frac{\epsilon}{\hat{\lambda}_1(2)}\right]||Du^+||^2_2\ \mbox{for some}\ c_4>0\\
		&&\mbox{(recall that $\alpha<\hat{\lambda}_1(2)$)}.
	\end{eqnarray*}
	
	Choosing $\epsilon\in(0,\hat{\lambda}_1(2)c_4)$, we have
	\begin{align*}
		& \varphi_+(u)\geq\frac{1}{p}||Du||^p_p\ \mbox{for all}\ u\in C^1_0(\overline\Omega), \mbox{with}\ ||u||_{C^1_0(\overline\Omega)}\leq\delta, \\
		\Rightarrow & u=0\ \mbox{is a local}\ C^1_0(\overline\Omega)\mbox{-minimizer of}\ \varphi_+, \\
		\Rightarrow & u=0\ \mbox{is a local}\ W^{1,p}_0(\Omega)\mbox{-minimizer of}\ \varphi_+\ \mbox{(see Proposition \ref{prop2})}.
	\end{align*}
The proof is complete.
\end{proof}

Now we are ready to produce constant sign smooth solutions for problem (\ref{eq1}).
\begin{prop}\label{prop11}
	If hypotheses $H(f)$ hold, then problem (\ref{eq1}) admits two constant sign smooth solutions
	$$u_0\in {\rm int}\,C_+\ \mbox{and}\ v_0\in -{\rm int}\,C_+.$$
\end{prop}
\begin{proof}
	We can easily check that $K_{\varphi_+}\subseteq C_+$. So, we may assume that $K_{\varphi_+}$ is finite or otherwise we already have an infinity of positive solutions for problem (\ref{eq1}). Then on account of Proposition \ref{prop10}, we can find small $\rho\in(0,1)$  such that
	\begin{equation}\label{eq50}
		0=\varphi_+(0)<\inf\left[\varphi_+(u): ||u||=\rho\right]=m_+
	\end{equation}
	(see Aizicovici, Papageorgiou \& Staicu \cite{1}, proof of Proposition 29). Hypotheses $H(f)(i),(ii)$ imply that given $\epsilon>0$, we can find $c_3>0$ such that
	\begin{equation}\label{eq51}
		F(z,x)\geq\frac{1}{p}[\eta(z)-\epsilon]x^p-c_5\ \mbox{for almost all}\ z\in\Omega,\ \mbox{for all}\ x\geq0.
	\end{equation}
	
	Then for $t>0$ we have
	\begin{eqnarray*}
		\varphi_+(t\hat{u}_1(p))\leq\frac{t^p}{p}\hat{\lambda}_1(p)+\frac{t^2}{2}||D\hat{u}_1(p)||^2_2 - \frac{t^p}{p}\int_\Omega\eta(z)\hat{u}_1(p)^pdz + \frac{t^p}{p}\epsilon + c_5|\Omega|_N \\
		\mbox{(see (\ref{eq51}) and recall that $||\hat{u}_1(p)||_p=1$)} \\
		=\frac{t^p}{p}\left[\int_\Omega[\hat{\lambda}_1(p)-\eta(z)]\hat{u}_1(p)^pdz+\epsilon\right]+
\frac{t^2}{2}||D\hat{u}_1(p)||^2_2+c_5|\Omega|_N.
	\end{eqnarray*}
	
	Since $\hat{u}_1(p)\in {\rm int}\,C_+$, we have
	$$k_0=\int_\Omega[\eta(z)-\hat{\lambda}_1(p)]\hat{u}_1(p)^pdz>0.$$
	
	Choosing $\epsilon\in(0,k_0)$, we have
	$$\varphi_+(t\hat{u}_1(p))\leq-\frac{t^p}{p}c_6 + \frac{t^2}{2}||D\hat{u}_1(p)||^2_2\ \mbox{for some $c_6>0$}.$$
	
	Since $p>2$, it follows that
	\begin{equation}\label{eq52}
		\varphi_+(t\hat{u}_1(p))\rightarrow-\infty\ \mbox{as}\ t\rightarrow+\infty.
	\end{equation}
	
	Also from Proposition \ref{prop7} we know that
	\begin{equation}\label{eq53}
		\varphi_+\ \mbox{satisfies the $C$-condition}.
	\end{equation}
	
	Then relations (\ref{eq50}), (\ref{eq52}), (\ref{eq53}) permit the use of Theorem \ref{th1} (the mountain pass theorem). So, we can find $u_0\in W^{1,p}_0(\Omega)$ such that
	\begin{align*}
		& u_0\in K_{\varphi_+}\ \mbox{and}\ \varphi_+(0)=0<m_+\leq\varphi_+(u_0), \\
		\Rightarrow & u_0\neq0.
	\end{align*}
	
	Since $K_{\varphi_+}\subseteq C_+$, we have that $u_0\in C_+\backslash\{0\}$. With $\rho=||u_0||_\infty$, let $\hat{\xi}_p>0$ be as postulated by hypothesis $H(f)(iv)$, that is, for almost all $z\in\Omega$, the function
	$$x\mapsto f(z,x) + \xi_p x^{p-1}$$
	is nondecreasing on $[0,\rho]$.
	
	Consider the map $a:\RR^N\rightarrow\RR^N$ defined by
	$$a(y)=|y|^{p-2}y+y\ \mbox{for all}\ y\in\RR^N.$$
	
	Then $a\in C^1(\RR^N,\RR^N)$ and
	\begin{align*}
		& {\rm div}\,a(Du)=\Delta_pu+\Delta u\ \mbox{for all}\ u\in W^{1,p}_0(\Omega), \\
		\Rightarrow & \nabla a(y) = |y|^{p-2}y\left[{\rm id}_N + (p-2)\frac{y\otimes y}{|y|^2}\right] + id_N
	\end{align*}
	with ${\rm id}_N$ being the identity map on $\RR^N$. For all $\xi\in\RR^N$ we have
	$$(\nabla a(y)\xi,\xi)_{\RR^N}\geq|\xi|^2>0\ \mbox{for all}\ \xi\in\RR^N\backslash\{0\}.$$
	
	Also, for $0\leq x\leq v\leq \rho$, we have
	\begin{align*}
		f(z,v)-f(z,x)\geq & -\xi_p(v^{p-1}-x^{p-1}) \\
		\geq & -\hat{\xi}_p(v-x)\ \mbox{for almost all}\ z\in\Omega,\ \mbox{for some}\ \hat{\xi_p}>0\ \mbox{(recall that $p>2$)}.
	\end{align*}
	
	So, we can apply the tangency principle of Pucci \& Serrin \cite[Theorem 2.5.2, p. 35]{27} and have
	$$u_0(z)>0\ \mbox{for all}\ z\in\Omega.$$
	
	Then the boundary point theorem of Pucci \& Serrin \cite[Theorem 5.5.1, p. 120]{27} implies that
	$$u_0\in {\rm int}\, C_+.$$
	
	By the Sobolev embedding theorem, $\varphi_-$ is sequentially weakly lower semicontinuous. Also, from Proposition \ref{prop8} we know that $\varphi_-$ is coercive. Hence by the Weierstrass-Tonelli theorem, we can find $v_0\in W^{1,p}_0(\Omega)$ such that
	\begin{align}
		& \varphi_-(v_0) = \inf[\varphi_-(u): u\in W^{1,p}_0(\Omega)], \label{eq54} \\
		\Rightarrow\ & v_0\in K_{\varphi_-}\subseteq -C_+. \nonumber
	\end{align}
	
	Hypotheses $H(f)(i),(ii),(iv)$ imply that given $\epsilon>0$, we can find $c_7=c_7(\epsilon)>0$ such that
	\begin{equation}\label{eq55}
		F(z,x)\geq\frac{1}{2}(\beta-\epsilon)x^2-c_7|x|^p\ \mbox{for almost all}\ z\in\Omega,\ \mbox{for all}\ x\leq0.
	\end{equation}
	Then for $t>0$ we have
	\begin{eqnarray*}
		\varphi_-(-t\hat{u}_1(2)) &\leq & \frac{t^p}{p}||D\hat{u}_1(2)||^p_p + \frac{t^2}{2}\hat{\lambda}_1(2) - \frac{t^2}{2}(\beta-\epsilon) + c_7t^p||\hat{u}_1(2)||^p_p \\
		&= & t^p \left[\frac{1}{p}||D\hat{u}_1(2)||^p_p + ||\hat{u}_1(2)||^p_p\right] - \frac{t^2}{2}[\beta-\epsilon-\hat{\lambda}_1(2)]\\
		&&\mbox{(see (\ref{eq55}) and recall that $||\hat{u}_1(2)||_2=1$)}.
	\end{eqnarray*}
	
	We choose $0<\epsilon<\beta-\hat{\lambda}_1(2)$ (see hypothesis $H(f)(iv)$). Then since $2<p$ for $t\in(0,1)$ small we can see that
	\begin{align*}
		& \varphi_-(-t\hat{u}_1(2))<0, \\
		\Rightarrow\ & \varphi_-(v_0)<0 = \varphi_-(0)\ \mbox{(see (\ref{eq54}))}, \\
		\Rightarrow\ & v_0 \neq0.
	\end{align*}
	
	Moreover, as for $u_0$, using the nonlinear strong maximum principle we have $v_0\in-{\rm int}\, C_+$ and this is the second constant sign solution of (\ref{eq1}).
\end{proof}

\section{Nodal Solutions -- Multiplicity Theorems}
In this section, using tools from Morse theory (critical groups), we show the existence of a nodal (sign changing) smooth solution and formulate our multiplicity theorems.

To produce a nodal sign, changing solution, we will need one more hypothesis which is the following one:

\smallskip
$H_0$: Problem (\ref{eq1}) has a finite number of solutions of constant sign.
\begin{remark}
	This condition is equivalent to saying that $K_{\varphi_+}$ and $K_{\varphi_-}$ are finite sets.
\end{remark}

We start by computing the critical groups of $\varphi$ at infinity.
\begin{prop}\label{prop12}
If hypotheses $H(f)$	 hold, then $C_k(\varphi,\infty)=0$ for all $k\in\NN_0.$
\end{prop}
\begin{proof}
	Let $\lambda>\hat{\lambda}_1(p),\ \lambda\not\in\hat{\sigma}(p)$ and consider the $C^1$-functional $\Psi:W^{1,p}_0(\Omega)\rightarrow\RR$ defined by
	\begin{equation}\label{eq56}
		\Psi(u)=\frac{1}{p}||Du||^p_p - \frac{\lambda}{p}||u^+||^p_p\ \mbox{for all}\ u\in W^{1,p}_0(\Omega).
	\end{equation}
	
	We consider the homotopy
	$$h(t,u) = (1-t)\varphi(u) + t\Psi(u)\ \mbox{for all}\ t\in[0,1],\ \mbox{for all}\ u\in W^{1,p}_0(\Omega).$$
	\begin{claim}\label{claim3}
		There exist $\gamma\in\RR$ and $\tau>0$ such that
		$$h(t,u)\leq\gamma\Rightarrow(1+||u||)||h'_u(t,u)||_*\geq\tau\ \mbox{for all}\ t\in[0,1].$$
	\end{claim}
	We argue by contradiction. Since $h(\cdot,\cdot)$ maps bounded sets to bounded sets, if the claim is not true, then we can find two sequences
	$\{t_n\}_{n\geq1}\subseteq[0,1]$ and $\{u_n\}_{n\geq1}\subseteq W^{1,p}_0(\Omega)$
	such that
	\begin{equation}\label{eq57}
		t_n\rightarrow t,\ ||u_n||\rightarrow\infty,\ h(t_n,u_n)\rightarrow-\infty\ \mbox{and}\ (1+||u_n||)h'_u(t_n,u_n)\rightarrow0.
	\end{equation}
	
	From the last convergence in (\ref{eq57}) we have
	\begin{eqnarray}
		|\langle A_p(u_n),h\rangle + (1-t)\langle A(u_n),h\rangle - (1-t_n)\int_\Omega f(z,u_n)hdz - \lambda t_n\int_\Omega(u^+_n)^{p-1}hdz| \nonumber \\
		\leq \frac{\epsilon_n||h||}{1+||u_n||} \label{eq58} \\
		\mbox{for all}\ h\in W^{1,p}_0(\Omega), \mbox{with}\ \epsilon_n\rightarrow0^+. \nonumber
	\end{eqnarray}
	
	From the third convergence in (\ref{eq57}), we see that we can find $n_0\in\NN$ such that
	\begin{eqnarray}
		||Du_n||^p_p + \frac{(1-t_n)p}{2}||Du_n||^2_2 - (1-t_n)\int_\Omega pF(z,u_n)dz - \lambda t_n||u^+_n||^p_p\leq-1 \label{eq59} \\
		\mbox{for all}\ n\geq n_0. \nonumber
	\end{eqnarray}
	
	In (\ref{eq58}) we choose $h=u_n\in W^{1,p}_0(\Omega)$. Then
	\begin{eqnarray}
		-||Du_n||^p_p - (1-t_n)||Du_n||^2_2 + (1-t_n)\int_\Omega f(z,u_n)u_ndz + \lambda t_n ||u^+_n||^p_p\leq\epsilon_n \label{eq60} \\
		\mbox{for all}\ n\in\NN. \nonumber
	\end{eqnarray}
	
	Adding (\ref{eq59}) and (\ref{eq60}), we obtain
	\begin{eqnarray}
		(1-t_n)\int_\Omega[f(z,u_n)u_n - pF(z,u_n)]dz\leq0\ \mbox{for all}\ n\geq n_1\geq n_0 \label{eq61} \\
		\mbox{(recall that $p>2$ and $\epsilon_n\rightarrow0^+$ as $n\rightarrow+\infty$).} \nonumber
	\end{eqnarray}
	
	We claim that $t<1$. If $t_n\rightarrow1$, then let $y_n=\frac{u_n}{||u_n||},\ n\in\NN$. We have $||y_n||=1$ for all $n\in\NN$ and so we may assume that
	\begin{equation}\label{eq62}
		y_n\xrightarrow{w}y\ \mbox{in}\ W^{1,p}_0(\Omega)\ \mbox{and}\ y_n\rightarrow y\ \mbox{in}\ L^p(\Omega).
	\end{equation}
	
	From (\ref{eq58}) we have
	\begin{eqnarray}
		|\langle A_p(y_n),h\rangle + \frac{1-t_n}{||u_n||^{p-2}}\langle A(y_n),h\rangle - (1-t_n)\int_\Omega\frac{N_f(u_n)}{||u_n||^{p-1}}hdz - \lambda t_n\int_\Omega(y^+_n)^{p-1}hdz| \nonumber \\
		\leq \frac{\epsilon_n||h||}{(1+||u_n||)||u_n||^{p-1}} \label{eq63} \\
		\mbox{for all}\ n\in\NN. \nonumber
	\end{eqnarray}
	
	In (\ref{eq63}) we choose $h=y_n-y$, pass to the limit as $n\rightarrow\infty$ and use (\ref{eq62}), (\ref{eq17}) and  $t_n\rightarrow 1,\ p>2$. Then
	\begin{eqnarray}\label{eq64}
		&& \lim_{n\rightarrow\infty}\langle A_p(y_n), y_n-y\rangle=0,\nonumber\\
		&\Rightarrow & y_n\rightarrow y\ \mbox{in}\ W^{1,p}_0(\Omega)\ \mbox{and so}\ ||y||=1\ \mbox{(see Proposition \ref{prop3}).}
	\end{eqnarray}
	
	So, if in (\ref{eq63}) we pass to the limit as $n\rightarrow\infty$ and use (\ref{eq64}), then
	\begin{equation}\label{eq65}
		\langle A_p(y),h\rangle=\lambda\int_\Omega(y^+)^{p-1}hdz\ \mbox{for all}\ h\in W^{1,p}_0(\Omega)
	\end{equation}
	(recall that $t_n\rightarrow1$). Choosing $h=-y^-\in W^{1,p}_0(\Omega)$, we have
	\begin{align*}
		& ||Dy^-||^p_p=0, \\
		\Rightarrow & y\geq0,\ y\neq0\ \mbox{(see (\ref{eq64}))}.
	\end{align*}
	
	From (\ref{eq65}) we have
	\begin{equation}\label{eq66}
		-\Delta_p y(z)=\lambda y(z)^{p-1}\ \mbox{for almost all}\ z\in\Omega,\ y|_{\partial\Omega}=0.
	\end{equation}
	
	Since $\lambda>\hat{\lambda}_1(p),\lambda\not\in\hat{\sigma}(p)$, from (\ref{eq65}) we infer that
	$$y=0,$$
	which contradicts (\ref{eq64}). Therefore $t<1$ and we have
	\begin{align}
		& \int_\Omega\left[f(z,u_n)u_n - pF(z,u_n)\right]dz\leq0\ \mbox{for all}\ n\geq n_1, \nonumber \\
		\Rightarrow & \int_\Omega\left[f(z,-u^-_n)(-u^-_n) - pF(z, -u^-_n)\right]dz\leq c_8 \label{eq67} \\
		& \mbox{for some}\ c_8>0, \mbox{all}\ n\geq n_1\ \mbox{(see hypothesis $H(f)(iii)$).} \nonumber
	\end{align}
	
	Using (\ref{eq67}) and reasoning as in claims \ref{claim1} and \ref{claim2} in the proof of Proposition \ref{prop6}, we establish that
	$$\{u_n\}_{n\geq1}\subseteq W^{1,p}_0(\Omega)\ \mbox{is bounded}.$$
	This contradicts (\ref{eq57}).
	Therefore we have proven the claim.
	
	Then the claim and Theorem 5.1.21 of Chang \cite[p. 334]{8} (see also Proposition 3.2 of Liang \& Su \cite{17}), we have
	\begin{align}
		& C_k(h(0,\cdot),\infty) = C_k(h(1,\cdot),\infty)\ \mbox{for all}\ k\in\NN_0, \nonumber \\
		\Rightarrow\ & C_k(\varphi,\infty)=C_k(\Psi,\infty)\ \mbox{for all}\ k\in\NN_0. \label{eq68}
	\end{align}
	
	So, our task is now to compute $C_k(\Psi,\infty)$. To this end, first note that since $\lambda>\hat{\lambda}_1(p)$, then
	\begin{equation}\label{eq69}
		K_\Psi=\{0\}.
	\end{equation}
	
	Consider the $C^1$-functional $\hat{h}:[0,1]\times W^{1,p}_0(\Omega)\rightarrow\RR$ defined by
	$$\hat h (t,u)=\Psi(u)-t\int_\Omega u(z)dz\ \mbox{for all}\ t\in[0,1], \ \mbox{for all}\ u\in W^{1,p}_0(\Omega).$$
	
	Suppose that $u\in K_{\hat{h}(t,\cdot)},\ t\in(0,1]$. Then
	\begin{equation}\label{eq70}
		\langle A_p(u),h\rangle = \lambda\int_\Omega(u^+)^{p-1}hdz + t\int_\Omega hdz\ \mbox{for all}\ h\in W^{1,p}_0(\Omega).
	\end{equation}
	
	Choosing $h=-u^-\in W^{1,p}_0(\Omega)$, we obtain
	\begin{align*}
		& ||Du^-||^p_p\leq 0, \\
		\Rightarrow\ & u\geq0,\ u\neq0.
	\end{align*}
	
	From (\ref{eq70}) we have
	\begin{align}
		& -\Delta_p u(z) = \lambda u(z)^{p-1} + t\ \mbox{for almost all}\ z\in\Omega,\ u|_{\partial\Omega}=0, \label{eq71} \\
		\Rightarrow & u\in {\rm int}\, C_+ \nonumber \\
		& \mbox{(by the nonlinear strong maximum principle, see \cite[p. 738]{12})}. \nonumber
	\end{align}
	
	Let $v\in {\rm int}\,C_+$ and consider the function
	$$R(v,u)(z) = |Dv(z)|^p - |Du(z)|^{p-2}(Du(z), D(\frac{v^p}{u^{p-1}})(z))_{\RR_N}.$$
	
	From the nonlinear Picone's identity of Allegretto \& Huang \cite{4} we have
	\begin{align*}
		0 \leq & \int_\Omega R(v,u)dz \\
		= & ||Dv||^p_p - \int_\Omega(-\Delta_p u)\frac{v^p}{u^{p-1}}dz \\
		& \mbox{(using the nonlinear Green's identity, see \cite[p. 211]{12}),} \\
		= & ||Dv||^p_p - \int_\Omega[\lambda u^{p-1}+t]\frac{v^p}{u^{p-1}}dz\ \mbox{(see (\ref{eq71}))} \\
		\leq & ||Dv||^p_p - \int_\Omega\lambda v^p dz\ \mbox{(since $u,v\in {\rm int}\, C_+$)}.
	\end{align*}
	
	Choosing $v=\hat{u}_1(p)\in {\rm int}\, C_+$, we have
	$$0\leq\hat{\lambda}_1(p)-\lambda<0\ \mbox{(recall that $||\hat{u}_1(p)||_p=1$),}$$
	a contradiction. Therefore
	\begin{equation}\label{eq72}
		K_{\hat{h}(t,\cdot)}=\emptyset\ \mbox{for all}\ t\in(0,1].
	\end{equation}
	
	From the homotopy invariance of singular homology, for $r>0$ small we have
	\begin{eqnarray}\label{eq73}
		H_k(\hat{h}(0,\cdot)^0\cap B_r, \hat{h}(0,\cdot)^0\cap B_r\backslash\{0\}) = H_k(\hat{h}(1,\cdot)^0\cap B_r, \hat{h}(1,\cdot)^0\cap B_r\backslash\{0\}) \\
		\mbox{for all}\ k\in\NN. \nonumber
	\end{eqnarray}
	
	Then (\ref{eq72}) and the noncritical interval theorem (see Chang \cite[Theorem 5.1.6, p. 320]{8} and Motreanu, Motreanu \& Papageorgiou \cite[Corollary 5.35, p. 115]{19}), we have
	\begin{equation}\label{eq74}
		H_k(\hat{h}(1,\cdot)^0\cap B_r, \hat{h}(1,\cdot)^0\cap B_r\backslash\{0\}) = 0\ \mbox{for all}\ k\in\NN_0.
	\end{equation}
	
	Also from the definition of critical groups, we have
	\begin{align}
		& H_k(\hat{h}(0,\cdot)^0\cap B_r, \hat{h}(0,\cdot)^0\cap B_r\backslash\{0\}) \nonumber \\
		= & H_k(\Psi^0\cap B_r, \Psi^0\cap B_r\backslash\{0\}) = C_k(\Psi, 0)\ \mbox{for all}\ k\in\NN_0. \label{eq75}.
	\end{align}
	
	From (\ref{eq73}), (\ref{eq74}), (\ref{eq75}) we conclude that
	\begin{align*}
		& C_k(\Psi,0)=0\ \mbox{for all}\ k\in\NN_0, \\
		\Rightarrow\ & C_k(\Psi, \infty)=0\ \mbox{for all}\ k\in\NN_0 \\
		& \mbox{(see (\ref{eq69}) and \cite[Proposition 6.61(c), p. 160]{19})} \\
		\Rightarrow\ & C_k(\varphi,\infty)=0\ \mbox{for all}\ k\in\NN_0\ \mbox{(see (\ref{eq68}))}. \qedhere
	\end{align*}
\end{proof}

Next, we compute the critical groups at infinity for the functional $\varphi_\pm$.
\begin{prop}\label{prop13}
	If hypotheses $H(f)$ hold, then $C_k(\varphi_+,\infty)=0$ for all $k\in\NN_0$.
\end{prop}
\begin{proof}
	Let $\Psi\in C^1(W^{1,p}_0(\Omega),\RR)$ be as in the proof of Proposition \ref{prop12} and consider the homotopy
	$$h_+(t,u)=(1-t)\varphi_+(u)+t\Psi(u)\ \mbox{for all}\ t\in[0,1],\ \mbox{for all}\ u\in W^{1,p}_0(\Omega).$$

	\begin{claim}\label{claim4}
		There exist $\gamma\in\RR$ and $\tau>0$ such that for all $t\in[0,1]$
		$$h_+(t,u)\leq\gamma\Rightarrow(1+||u||)||(h_+)'(t,u)||_*\geq\tau.$$
	\end{claim}
	As in the proof of Proposition \ref{prop12}, we argue by contradiction. So, we can find two sequences
	$$\{t_n\}_{n\geq1}\subseteq[0,1]\ \mbox{and}\ \{u_n\}_{n\geq1}\subseteq W^{1,p}_0(\Omega)$$
	such that
	\begin{equation}\label{eq76}
		t_n\rightarrow t, ||u_n||\rightarrow\infty, h_+(t_n,u_n)\rightarrow-\infty\ \mbox{and}\ (1+||u_n||)(h_+)'(t_n,u_n)\rightarrow0.
	\end{equation}
	
	From the last convergence in (\ref{eq76}), we have
	\begin{eqnarray}\label{eq77}
		|\langle A_p(u_n),h\rangle + (1-t_n)\langle A(u_n),h\rangle - (1-t_n)\int_\Omega f_+(z,u_n)hdz -\lambda t_n\int_\Omega(u^+_n)^{p-1}hdz| \nonumber \\
		\leq \frac{\epsilon_n||h||}{1+||u_n||}\ \mbox{for all}\ h\in W^{1,p}_0(\Omega)\ \mbox{with}\ \epsilon_n\rightarrow0^+.
	\end{eqnarray}
	
	In (\ref{eq77}) we choose $h=-u^-_n\in W^{1,p}_0(\Omega)$. Then
	\begin{eqnarray}\label{eq78}
		&& ||Du^-_n||^p_p + (1-t)||Du^-_n||^2_2 \leq\epsilon_n\ \mbox{for all}\ n\in\NN, \nonumber \\
		&\Rightarrow & u^-_n\rightarrow0\ \mbox{in}\ W^{1,p}_0(\Omega).
	\end{eqnarray}
	
	From (\ref{eq77}) and (\ref{eq78}) we infer that
	\begin{eqnarray*}
		|\langle A_p(u^+_n),h\rangle + (1-t_n)\langle A(u^+_n),h\rangle - (1-t_n)\int_\Omega f(z,u^+_n)hdz -\lambda t_n\int_\Omega(u^+_n)^{p-1}hdz| \\
		\leq \epsilon'_n||h||\ \mbox{for all}\ h\in W^{1,p}_0(\Omega)\ \mbox{with}\ \epsilon'_n\rightarrow0^+.
	\end{eqnarray*}
	
	Suppose that $||u^+_n||\rightarrow+\infty$. We set $v_n=\frac{u^+_n}{||u^+_n||},\ n\in\NN$. Then $||v_n||=1,\ v_n\geq0$ for all $n\in\NN$ and so we may assume that
	$$v_n\xrightarrow{w}v\ \mbox{in}\ W^{1,p}_0(\Omega)\ \mbox{and}\ v_n\rightarrow v\ \mbox{in}\ L^p(\Omega).$$
	
	Reasoning as in the proof of Proposition \ref{prop6} (see the proof of Claim \ref{claim2}), we reach a contradiction and so we infer that
	\begin{align*}
		& \{u^+_n\}_{n\geq1}\subseteq W^{1,p}_0(\Omega)\ \mbox{is bounded}, \\
		\Rightarrow\ & \{u_n\}_{n\geq1}\subseteq W^{1,p}_0(\Omega)\ \mbox{is bounded (see (\ref{eq78}))}.
	\end{align*}
	
	But this contradicts (\ref{eq76}). Hence the claim holds and as before (see the proof of Proposition \ref{prop12}) we have
	\begin{align*}
		& C_k(h_+(0,\cdot),\infty) = C_k(h_+(1,\cdot),\infty)\ \mbox{for all}\ k\in\NN_0, \\
		\Rightarrow & C_k(\varphi_+,\infty) = C_k(\Psi,\infty)\ \mbox{for all}\ k\in\NN_0 \\
		\Rightarrow & C_k(\varphi_+,\infty) = 0\ \mbox{for all}\ k\in\NN_0 \\
		& \mbox{(see the end of the proof of Proposition \ref{prop12})}.
	\end{align*}
The proof is complete.
\end{proof}

\begin{prop}\label{prop14}
	If hypotheses $H(f)$ hold, then $C_k(\varphi_-,\infty)=\delta_{k,0}\ZZ$ for all $k\in\NN_0$.
\end{prop}
\begin{proof}
	From Proposition \ref{prop8} we know that $\varphi_-$ is coercive. So, it is bounded below and satisfies the $C$-condition (see Corollary \ref{corollary9}). Hence Proposition 6.64(a) of Motreanu, Motreanu \& Papageorgiou \cite[p. 116]{19}, implies that
	$$C_k(\varphi_-,\infty)=\delta_{k,0}\ZZ\ \mbox{for all}\ k\in\NN_0.$$
\end{proof}

Next, we compute the critical groups of $\varphi$ at $u=0$.
\begin{prop}\label{prop15}
	If hypotheses $H(f)$ hold, then $C_k(\varphi,0)=\delta_{k,0}\ZZ$ for all $k\in\NN_0$.
\end{prop}
\begin{proof}
	Let $\alpha\in(0,\hat{\lambda}_1(2))$ and $\beta\in(\hat{\lambda}_1(2),\hat{\lambda}_2(2))$ be as postulated by hypothesis $H(f)(iv)$. We consider the $C^1$-functional $\hat{\Psi}_0:H^1_0(\Omega)\rightarrow\RR$ defined by
	$$\hat{\Psi}_0(u)=\frac{1}{2}||Du||^2_2 - \frac{\alpha}{2}||u^+||^2_2 - \frac{\beta}{2}||u^-||^2_2\quad \mbox{for all}\ u\in H^1_0(\Omega).$$
	
	Let $\Psi_0=\hat{\Psi}_0|_{W^{1,p}_0(\Omega)}$ (recall that $2<p$) and consider the homotopy
	$$h(t,u)=(1-t)\varphi(u) + t\Psi_0(u)\ \mbox{for all}\ t\in[0,1],\ \mbox{for all}\ u\in W^{1,p}_0(\Omega).$$
	
	Suppose that we can find $\{t_n\}_{n\geq1}\subseteq[0,1]$ and $\{u_n\}_{n\geq1}\subseteq W^{1,p}_0(\Omega)$ such that
	\begin{equation}\label{eq79}
		t_n\rightarrow t\in[0,1],\ u_n\rightarrow0\ \mbox{in}\ W^{1,p}_0(\Omega)\ \mbox{and}\ h'_u(t_n,u_n)=0\ \mbox{for all}\ n\in\NN.
	\end{equation}
	
	From the equality in (\ref{eq79}) we have
	\begin{eqnarray}
		(1-t_n)A_p(u_n)+A(u_n) = (1-t_n)N_f(u_n) + t_n[\alpha u^+_n - \beta u^-_n]\ \mbox{in}\ W^{-1,p'}(\Omega), \label{eq80} \\
		\mbox{for all}\ n\in\NN. \nonumber
	\end{eqnarray}
	
	Let $y_n=\frac{u_n}{||u_n||},\ n\in\NN$. Then $||y_n||=1$ for all $n\in\NN$ and so we may assume that
	\begin{equation}\label{eq81}
		y_n\xrightarrow{w}y\ \mbox{in}\ W^{1,p}_0(\Omega)\ \mbox{and}\ y_n\rightarrow y\ \mbox{in}\ L^p(\Omega).
	\end{equation}
	
	From (\ref{eq80}) we have
	\begin{equation}\label{eq82}
		(1-t_n)||u_n||^{p-2} A_p(y_n) + A(y_n) = (1-t_n)\frac{N_f(u_n)}{||u_n||} + t_n[\alpha y^+_n - \beta y^-_n]\ \mbox{for all}\ n\in\NN.
	\end{equation}
	
	Note that hypotheses $H(f)(i),(ii),(iv)$ imply that for some $c_9>0$
	\begin{align*}
		& |f(z,x)| \leq c_9[|x| + |x|^{p-1}]\ \mbox{for almost all}\ z\in\Omega,\ \mbox{for all}\ x\in\RR, \\
		\Rightarrow & \left\{\frac{N_f(u_n)}{||u_n||}\right\}_{n\geq 1}\subseteq L^{p'}(\Omega)\ \mbox{is bounded}.
	\end{align*}
	
	This fact and hypothesis $H(f)(iv)$ imply that
	\begin{equation}\label{eq83}
		\frac{N_f(u_n)}{||u_n||}\xrightarrow{w} \alpha y^+ - \beta y^-\ \mbox{in}\ L^{p'}(\Omega)\ \mbox{as}\ n\rightarrow\infty
	\end{equation}
	(see Aizicovici, Papageorgiou \& Staicu \cite{1}, proof of Proposition 16). From (\ref{eq82}) we have
	\begin{align*}
		& -(1-t_n)||u_n||^{p-2}\Delta_p y_n(z)-\Delta y_n(z) \\
		= & (1-t_n)\frac{f(z,u_n(z))}{||u_n||} + t_n[\alpha y^+_n(z) - \beta y^-_n(z)]\ \mbox{for almost all}\ z\in\Omega,\ y_n|_{\partial\Omega}=0, n\in\NN.
		\end{align*}
		
		Then (\ref{eq79}), (\ref{eq81}), (\ref{eq83}) (recall that $p>2$) and Theorem 7.1 of Ladyzhenskaya \& Uraltseva \cite[p. 286]{16}, we can find $M_8>0$ such that
		$$||y_n||_\infty\leq M_8\ \mbox{for all}\ n\in\NN.$$
		
		Then invoking Theorem 1 of Lieberman \cite{18}, we can find $\vartheta\in(0,1)$ and $M_9>0$ such that
		$$y_n\in C^{1,\vartheta}_0(\overline\Omega)\ \mbox{all}\ ||y_n||_{C^{1,\vartheta}_0(\overline\Omega)}\leq M_9\ \mbox{for all}\ n\in\NN.$$
		
		From (\ref{eq61}) and the compact embedding of $C^{1,\vartheta}_0(\overline\Omega)$ into $C^1_0(\overline\Omega)$, we have
		\begin{eqnarray}\label{eq84}
			& y_n\rightarrow y\ \mbox{in}\ C^1_0(\overline\Omega), \\
			\Rightarrow & ||y||=1\ \mbox{and so}\ y\neq0.\nonumber
		\end{eqnarray}
		
		If in (\ref{eq82}) we pass to the limit as $n\rightarrow\infty$ and use (\ref{eq79}), (\ref{eq83}), (\ref{eq84}) and the fact that $p>2$, we obtain
		\begin{align*}
			& A(y) = \alpha y^+ - \beta y^-\ \mbox{in}\ W^{-1,p'}(\Omega), \\
			\Rightarrow & -\Delta y(z) = \alpha y^+(z) - \beta y^-(z)\ \mbox{for almost all}\ z\in\Omega,\ y|_{\partial\Omega}=0.
		\end{align*}
		
		Since $0<\alpha<\hat{\lambda}_1(2)<\beta<\hat{\lambda}_2(2)$ (see hypothesis $H(f)(iv)$), $(\alpha,\beta)\not\in\Sigma_2$ and so
		$$y=0,$$
		a contradiction (recall that $||y||=1$). Therefore (\ref{eq79}) cannot occur and then the homotopy invariance of critical groups (see Gasinski \& Papageorgiou \cite[Theorem 5.125, p. 836]{14}) implies that
		\begin{equation}\label{eq85}
			C_k(\varphi,0) = C_k(\Psi_0,0)\ \mbox{for all}\ k\in\NN_0.
		\end{equation}
		
		Since $W^{1,p}_0(\Omega)$ is dense in $H^1_0(\Omega)$ from Palais \cite[Theorem 16]{20} (see also Chang \cite[p. 14]{7}), we have
		\begin{equation}\label{eq86}
			C_k(\Psi_0,0)=C_k(\hat{\Psi}_0,0)\ \mbox{for all}\ k\in\NN_0.
		\end{equation}
		
		But Theorem 1.1(a) of Perera \& Schechter \cite{26} implies that
		\begin{align*}
			& C_k(\hat{\Psi}_0,0)=\delta_{k,0}\ZZ\ \mbox{for all}\ k\in\NN_0, \\
			\Rightarrow & C_k(\varphi,0) = \delta_{k,0}\ZZ\ \mbox{for all}\ k\in\NN_0\ \mbox{(see (\ref{eq85}), (\ref{eq86})).}
		\end{align*}
\end{proof}
Also, we have the following property.
\begin{prop}\label{prop16}
	If hypotheses $H(f)$ hold, then $C_k(\varphi_-,0)=0$ for all $k\in\NN_0$.
\end{prop}
\begin{proof}
	In this case we consider the $C^1$-functional $\hat{\Psi}:H^1_0(\Omega)\rightarrow\RR$ defined by
	$$\hat{\Psi}_-(u) = \frac{1}{2}||Du||^2_2 - \beta||u^-||^2_2\ \mbox{for all}\ u\in H^1_0(\Omega).$$
	
	We set $\Psi_-=\hat{\Psi}_-|_{W^{1,p}_0(\Omega)}$ (recall that $p>2$) and consider the homotopy
	$$h_-(t,u)=(1-t)\varphi_-(u) + t\Psi_-(u)\ \mbox{for all}\ t\in[0,1],\ \mbox{for all}\ u\in W^{1,p}_0(\Omega).$$
	
	As in the proof of Proposition \ref{prop15}, via the homotopy invariance of critical groups, we have
	$$C_k(\varphi_-,0)=C_k(\Psi_-,0)\ \mbox{for all}\ k\in\NN_0.$$
	
	Recalling that the nonprincipal eigenfunction of $(-\Delta, H^1_0(\Omega))$ are nodal and since $\beta>\hat{\lambda}_1(2)$, we infer that $K_{\Psi_-}=\{0\}$. Moreover, as in the last part of the proof of Proposition \ref{prop12}, using Picone's identity, we have
	\begin{eqnarray*}
	&&C_k(\Psi_-,0)=0\ \mbox{for all}\ k\in\NN_0,\\
	&\Rightarrow & C_k(\varphi_-,0)=0\ \mbox{for all}\ k\in\NN_0.
	\end{eqnarray*}
This completes the proof.
\end{proof}

Now we are ready for our first multiplicity theorem. Recall that at the beginning of this section, we have introduced  an extra hypothesis $H_0$, which says that the constant sign solutions of (\ref{eq1}) are finite. This is equivalent to saying that
\begin{align*}
	& K_{\varphi_-}=\{v_i\}^m_{i=1}\cup\{0\}\subseteq(-{\rm int}\, C_+)\cup\{0\}, \\
	& K_{\varphi_+}=\{u_l\}^n_{l=1}\cup\{0\}\subseteq {\rm int}\, C_+\cup\{0\}.
\end{align*}

From Proposition \ref{prop11} we know that
$$m,n\in\NN.$$

Then we have the following multiplicity theorem.
\begin{theorem}\label{th17}
	If hypotheses $H(f),H_0$ hold, then problem (\ref{eq1}) has at least three nontrivial smooth solutions
	$$u_0\in {\rm int}\, C_+,\ v_0\in-{\rm int}\, C_+\ \mbox{and}\ y_0\in C^1_0(\overline\Omega)\ \mbox{nodal}.$$
\end{theorem}
\begin{proof}
	From Proposition \ref{prop11} we already have two nontrivial constant sign smooth solutions
	$$u_0\in {\rm int}\, C_+\ \mbox{and}\ v_0\in- {\rm int}\, C_+.$$
	
	By hypothesis $H_0$ we have
	\begin{align*}
		& K_{\varphi_-}=\{v_i\}^m_{i=1}\cup\{0\}\subseteq(-{\rm int}\, C_+)\cup\{0\}, \\
		& K_{\varphi_+}=\{u_l\}^n_{l=1}\cup\{0\}\subseteq {\rm int}\, C_+\cup\{0\}.
	\end{align*}
	
	We set
	\begin{align*}
		& \chi_-(v_i)=\sum_{k\geq0}(-1)^k {\rm rank}\, C_k(\varphi_-, v_i),\quad \ i=1,\dots,m, \\
		& \chi_+(u_l)=\sum_{k\geq0}(-1)^k {\rm rank}\, C_k(\varphi_+, u_l),\quad \ l=1,\dots,n, \\
		& \chi(v_i)=\sum_{k\geq0}(-1)^k {\rm rank}\, C_k(\varphi, v_i) \\
		& \chi(u_l)=\sum_{k\geq0}(-1)^k {\rm rank}\, C_k(\varphi, u_l).
	\end{align*}
	
	Since $v_i\in- {\rm int}\, C_+$ and $u_l\in {\rm int}\, C_+$, we have
	\begin{equation}\label{eq87}
		\chi_-(v_i) = \chi(v_i)\ \mbox{for all}\ i=1,\dots,m\ \mbox{and}\ \chi_+(u_l) = \chi(u_l)\ \mbox{for all}\ l=1,\dots,n.
	\end{equation}
	
	From Propositions \ref{prop10}, \ref{prop15}, \ref{prop16}, we have
	\begin{equation}\label{eq88}
		\chi_+(0)=1,\ \chi(0)=1,\ \chi_-(0)=0.
	\end{equation}
	
	Let $\{y_i\}^d_{j=1}\subseteq K_\varphi\subseteq C^1_0(\overline\Omega)$ (nonlinear regularity theory) be the set of nodal solutions of (\ref{eq1}) (if there are no nodal solutions, then $d=0$). From the Morse relation (see (\ref{eq7})), we have
	\begin{align}
		& \chi_-(0) + \sum^m_{i=1}\chi_-(v_i) = 0+\sum^m_{i=1}\chi(v_j)=1\ \mbox{(see (\ref{eq87}), (\ref{eq88}) and Proposition \ref{prop14}),} \label{eq89} \\
		& \chi_+(0) + \sum^n_{l=1}\chi_+(u_l) = 1+\sum^n_{l=1}\chi(u_l)=0\ \mbox{(see (\ref{eq87}), (\ref{eq88}) and Proposition \ref{prop13}),} \label{eq90} \\
		& \chi(0) + \sum^m_{i=1}\chi(v_i) + \sum^n_{l=1}\chi(u_l) + \sum^d_{j=1}\chi(y_j)=0\ \mbox{(see Proposition \ref{prop12})}, \nonumber \\
		\Rightarrow & 1 + 1 - 1 + \sum^d_{j=1}\chi(y_j)=0\ \mbox{(see (\ref{eq88}), (\ref{eq89}), (\ref{eq90}))}, \nonumber\\
		\Rightarrow & \sum^d_{j=1}\chi(y_j)=-1\ \mbox{and so}\ d\geq1.\nonumber
	\end{align}
	
	This means that problem (\ref{eq1}) admits at least one nodal solution $y_0\in C^1_0(\overline\Omega)$.
\end{proof}
We can drop hypothesis $H_0$ at the expense of strengthening the regularity of $f(z,\cdot)$. Then we can still have a three nontrivial solutions multiplicity theorem, but without providing sign information about the third nontrivial smooth solution.

The new hypotheses on $f(z,x)$ are the following:

\smallskip
$H(f)'$: $f:\Omega\times\RR\rightarrow\RR$ is a Carath\'eodory function such that for almost all $z\in\Omega$, $f(z,0)=0$, $f(z,\cdot)\in C^1(\RR\backslash\{0\})$ and hypotheses $H(f)'(i)\rightarrow(iv)$ are the same as the corresponding hypotheses $H(f)(i)\rightarrow(iv)$
\begin{theorem}\label{th18}
	If hypotheses $H(f)'$ hold, then problem (\ref{eq1}) admits at least three nontrivial smooth solutions
	$$v_0\in {\rm int}\, C_+,\ v_0\in -{\rm int}\, C_+\ \mbox{and}\ y_0\in C^1_0(\overline\Omega).$$
\end{theorem}
\begin{proof}
	Again, from Proposition \ref{prop11} we already have two nontrivial constant sign smooth solutions
	$$u_0\in {\rm int}\,C_+\ \mbox{and}\ v_0\in -{\rm int}\,C_+.$$
	
	From the proof of Proposition \ref{prop11} we know that $u_0\in K_{\varphi_+}$ is of mountain pass type. Since $u_0\in {\rm int}\,C_+$, we have
	\begin{align*}
		& C_k(\varphi, u_0) = C_k(\varphi_+, u_0)\ \mbox{for all}\ k\in\NN \\
		\mbox{and}\ & C_1(\varphi_+,u_0)\neq0 \\
		& \mbox{(see Motreanu, Motreanu \& Papageorgiou \cite[Corollary 6.81, p. 168]{19})}.
	\end{align*}
	
	Therefore
	\begin{equation}\label{eq91}
		C_1(\varphi, u_0)\neq0.
	\end{equation}
	
	But $\varphi\in C^2(W^{1,p}_0(\Omega)\backslash\{0\})$. Hence from (\ref{eq91}) and Papageorgiou \& R\u{a}dulescu \cite{21} we have
	\begin{equation}\label{eq92}
		C_k(\varphi,u_0)=\delta_{k,1}\ZZ\ \mbox{for all}\ k\in\NN_0.
	\end{equation}
	
	The negative solution $v_0\in K_{\varphi_-}$ is a global minimizer of $\varphi_-$. Since $\varphi_-|_{-C_+}=\varphi|_{-C_+}$ and $v_0\in -{\rm int}\,C_+$, it follows that $v_0$ is a local $C^1_0(\overline\Omega)$-minimizer of $\varphi$. Invoking Proposition \ref{prop2}, we infer that $v_0$ is a local $W^{1,p}_0(\Omega)$-minimizer of $\varphi$. Therefore
	\begin{equation}\label{eq93}
		C_k(\varphi,v_0)=\delta_{k,0}\ZZ\ \mbox{for all}\ k\in\NN_0.
	\end{equation}
	
	From Propositions \ref{prop12} and \ref{prop15}, we have
	\begin{align}
		& C_k(\varphi,\infty)=0\ \mbox{for all}\ k\in\NN_0, \label{eq94} \\
		& C_k(\varphi, 0)=\delta_{k,0}\ZZ\ \mbox{for all}\ k\in\NN_0. \label{eq95}
	\end{align}
	
	Suppose that $K_\varphi=\{u_0,v_0,0\}$. Then the Morse relation with $t=-1$ and (\ref{eq92}), (\ref{eq93}), (\ref{eq94}), (\ref{eq95}) imply that
	\begin{align*}
		& (-1)^1 + (-1)^0 + (-1)^0 = 0, \\
		\Rightarrow & (-1)^1=0,\ \mbox{a contradiction}.
	\end{align*}
	
	So, there exists $y_0\in K_\varphi,\ y_0\not\in\{u_0,v_0,0\}$. Hence $y_0$ is the third nontrivial solution of problem (\ref{eq1}) and the nonlinear regularity theory implies that $y_0\in C^1_0(\overline\Omega).$
\end{proof}

\medskip
{\bf Acknowledgments.} This research was supported by the Slovenian Research Agency grants P1-0292, J1-8131, J1-7025, and N1-0064. V.D. R\u adulescu acknowledges the support through a grant of the Ministry of Research and Innovation, CNCS--UEFISCDI, project number PN-III-P4-ID-PCE-2016-0130, within PNCDI III.

\end{document}